\documentclass[a4paper,12pt,reqno]{amsart}
\usepackage{amssymb}
\usepackage{amsmath}
\usepackage{enumitem}
\usepackage{mathrsfs}
\usepackage[all]{xy}
\usepackage{mathtools}
\usepackage{hyperref}
\usepackage{url}

\usepackage[OT2,T1]{fontenc}
\DeclareSymbolFont{cyrletters}{OT2}{wncyr}{m}{n}
\usepackage[british]{babel}

\usepackage[margin=1.3in]{geometry}
\numberwithin{equation}{section} \numberwithin{figure}{section}

\DeclareMathOperator{\Pic}{Pic} 
\DeclareMathOperator{\Gal}{Gal} \DeclareMathOperator{\NS}{NS}
 
\DeclareMathOperator{\Spec}{Spec}
 
\DeclareMathOperator{\Hom}{Hom} 
\DeclareMathOperator{\im}{Im}   
\DeclareMathOperator{\vol}{vol} 
 \DeclareMathOperator{\Val}{Val}

\DeclareMathOperator{\Br}{Br}

\DeclareMathOperator{\Norm}{N}  
\DeclareMathOperator{\supp}{supp}
\DeclareMathSymbol{\Sha}{\mathalpha}{cyrletters}{"58}

\newcommand{\OO}{\mathcal{O}}
\newcommand{\oo}{\mathfrak{o}}

\newcommand{\cP}{\mathcal{P}}

\newcommand{\x}{\mathbf{x}}
\newcommand{\y}{\mathbf{y}}
\renewcommand{\d}{\mathrm{d}}
\renewcommand{\c}{\mathbf{c}}
\renewcommand{\v}{\mathbf{v}}
\renewcommand{\u}{\mathbf{u}}
\newcommand{\z}{\mathbf{z}}

\renewcommand{\a}{\mathbf{a}}

\renewcommand{\AA}{\mathbb{A}}
\newcommand{\A}{\mathbf{A}}
\newcommand\FF{\mathbb{F}}
\newcommand\PP{\mathbb{P}}
\newcommand\ZZ{\mathbb{Z}}
\newcommand\NN{\mathbb{N}}
\newcommand\QQ{\mathbb{Q}}
\newcommand\RR{\mathbb{R}}

\newcommand\GG{\mathbb{G}}

\newcommand\Gm{\GG_\mathrm{m}}
\newcommand{\Adele}{\mathbf{A}}

\newcommand{\HH}{\mathrm{H}}

\newcommand{\fo}{\mathfrak{o}}
\newcommand{\fp}{\mathfrak{p}}
\newcommand{\fa}{\mathfrak{a}}
\newcommand{\ve}{\varepsilon}

\renewcommand{\hat}{\widehat}

\renewcommand{\leq}{\leqslant}

\renewcommand{\geq}{\geqslant}

\newcommand{\sumstar}{\sideset{}{^*}\sum}

\newcommand{\0}{\mathbf{0}}

\usepackage[usenames,dvipsnames]{color}

\newtheorem{lemma}{Lemma}

\newtheorem{theorem}[lemma]{Theorem}
\newtheorem{proposition}[lemma]{Proposition}
\newtheorem{corollary}[lemma]{Corollary}

\theoremstyle{definition}
\newtheorem{example}[lemma]{Example}
\newtheorem{definition}[lemma]{Definition}
\newtheorem{remark}[lemma]{Remark}

\numberwithin{lemma}{section}

\begin{document}

\title[Sieving rational points]
{Sieving rational points on varieties}

\author{\sc Tim Browning}
\address{School of Mathematics\\
University of Bristol\\ Bristol\\ BS8 1TW
\\ UK}
\email{t.d.browning@bristol.ac.uk}

\author{\sc Daniel Loughran}
\address{
School of Mathematics\\
University of Manchester \\
Oxford Road\\
Manchester\\
M13 9PL\\
UK}
\email{daniel.loughran@manchester.ac.uk}

\subjclass[2010]
{14G05; 
11N36, 
11P55, 
14D10. 
}

\begin{abstract}
An upper bound 
sieve for rational points on suitable varieties is developed, together with applications to counting rational points in thin sets, 
local solubility in families, 
and to the notion of 
``friable'' rational points with respect	to divisors.
In the special case of quadrics, sharper estimates are obtained by developing a version of the Selberg sieve for rational points. 
\end{abstract}

\date{\today}

\maketitle

\thispagestyle{empty}

\tableofcontents

\section{Introduction}

Sieves are a ubiquitous tool in analytic number theory and have numerous applications. Typically, 
one is given a subset $\Omega_p \subset \ZZ/p\ZZ$ for each prime $p$  and the challenge is to count the number of integers $n$ in an interval 
for which $n \bmod p \in \Omega_p$ for all $p$. 
In favourable situations one can deduce  asymptotic formulae from suitable equidistribution statements.  In this paper, however,  our focus is on upper bound sieves. These can  be obtained through a variety of means, the most successful being  variants of  the large sieve or the
Selberg sieve, as explained in \cite[Chapters 7--9]{FI}. 

The above set-up can be generalised in many ways, such as in the
abstract version of the large sieve developed by Kowalski 
 \cite[\S2.1]{Kow08}, for example.    In our investigation we adopt the 
 following approach: one is given a smooth projective variety $X$  over a number
field $k$, together with a height function $H$ and a model $\mathcal{X}$ over 
the ring of integers $\fo_k$ of $k$, and for each non-zero prime ideal $\fp$ of $\oo_k$ a subset $\Omega_\fp \subset \mathcal{X}(\oo_k/\fp)$.
The goal is to  obtain upper bounds for
$$
	\# \{ x \in X(k) : H(x) \leq B, ~x \bmod \fp \in \Omega_{\fp} \mbox{ for all } \fp\}.
$$
We adopt two points of view in 
addressing this  counting problem. First we see how much can be achieved by working in as general
a set-up as possible. The set-up we take is that of  varieties whose rational points are \emph{equidistributed}
with respect to a suitable adelic Tamagawa measure, a property that allows us to  sieve by any finite list of local conditions. Given the generality we work in, we are only able to obtain little $o$-results here, rather than precise upper bounds. Next, by specialising to the case of quadric hypersurfaces, we use the Hardy--Littlewood circle method to develop a version of the Selberg sieve for quadrics, which ultimately gives explicit  upper bounds.

\subsection{Equidistribution and sieving rational points} \label{sec:equi_intro}

\subsubsection{Manin's conjecture and equidistribution}
We begin by recalling Manin's conjecture \cite{FMT},  \cite{BM90}, \cite[\S3]{Pey03}.
We work with the following classes of varieties.

\begin{definition}
	A smooth projective geometrically integral variety $X$ over a field $k$ is called \emph{almost Fano}
	if 
	\begin{itemize}
		\item $\HH^1(X, \OO_X) = \HH^2(X,\OO_X) = 0$;
		\item The geometric Picard group $\Pic \bar{X}$ is torsion free;
		\item The anticanonical divisor $-K_X$ is big.
	\end{itemize}
\end{definition}
Let $X$ be an almost Fano variety over a number field $k$
and $H$ an anticanonical height function on $X$ (that is, a height function associated to
a choice of adelic metric on the anticanonical bundle of $X$). 
If $X(k) \neq \emptyset$, 
Manin's original conjecture predicts the existence of Zariski 
open subset $U \subset X$ such that
\begin{equation} \label{conj:Manin}
	N(U,H,B) := \#\{ x \in U(k) : H(x) \leq B\} \sim c_{X,H} B (\log B)^{\rho(X) - 1},
\end{equation}
where $\rho(X)$ is the rank of the Picard group of $X$ and $c_{X,H} > 0$.
The leading constant $c_{X,H}$ in \eqref{conj:Manin} has a conjectural interpretation
due to Peyre \cite{Pey95}, which is expressed in terms of a certain Tamagawa measure on
the adelic space $X(\Adele_k)$.

For our first results we assume that the rational points of bounded height 
are \emph{equidistributed}. Intuitively, 
this means that conditions imposed at finitely many different
places are asymptotically independent, and alter the leading constant in \eqref{conj:Manin}
by the Tamagawa measure of the imposed conditions. We recall the relevant definitions in \S \ref{sec:Tamagawa}. This property, first introduced to the subject by 
Peyre \cite[\S3]{Pey95}, is very natural; Peyre
showed that it holds if \eqref{conj:Manin} holds with Peyre's constant
with respect to all choices of anticanonical height function.

The equidistribtion property is known to hold for the following classes of almost Fano varieties:
Flag varieties \cite[\S6.2.4]{Pey95}, toric
varieties \cite[\S3.10]{CLT17}, equivariant compactifications of additive groups \cite[Rem.~0.2]{CLT02}, and complete intersections in many variables
(proved over $\QQ$ in \cite[Prop.~5.5.3]{Pey95}; the result over general number fields is obtained by modifying the arguments given in \cite[\S4.3]{Lou15}).

The equidistribution property trivially allows one to sieve with respect to finitely many primes.
One can use it to give upper bounds for sieving with respect to infinitely many primes by taking 
the limit over the conditions. 

\subsubsection{Thin sets}
The original version of Manin's conjecture \eqref{conj:Manin} is false in general, as first
shown in \cite{BT96}. The problem is that  the union of 
the accumulating subvarieties in $X$ can be Zariski dense, so that there is no sufficiently small open set $U\subset X$
on  which the expected asymptotic formula holds.

Numerous authors have recently investigated a ``thin'' version of Manin's conjecture (see  \cite[\S8]{Pey03}, \cite{LeR13}, \cite{BL17}  or \cite{LT17}),
 where one is allowed to remove a thin subset of $X(k)$, rather than just a Zariski closed set.
(We use the term \emph{thin set} in the sense of Serre \cite[\S3.1]{Ser08};  the 
various definitions are  recalled
in \S\ref{sec:thin}.)

A natural question is whether removing a thin subset could change the asymptotic behaviour of the counting function $N(U,H,B)$.
We show that this is not the case when the rational points are equidistributed.

\begin{theorem} \label{thm:equi_thin}
	Let $X$ be an almost Fano variety over $k$ and $H$ an anticanonical height function on $X$.
	Assume that the rational points are equidistributed on some dense open subset $U \subset X$.
	Let $\Upsilon \subset U(k)$ be thin. Then 
	$$\lim_{B \to \infty} \frac{\#\{ x \in \Upsilon : H(x) \leq B\}}{N(U,H,B)} = 0.$$	
\end{theorem}

Theorem \ref{thm:equi_thin}  recovers the well-known fact that a thin subset of $\PP^n(k)$ contains only $0\%$ of the total number of rational points of $\PP^n(k)$, when ordered by height. This special case is due to Cohen  \cite{cohen} and Serre  
 \cite[Thm.~13.3]{Ser97}.
 
\subsubsection{Fibrations}

Given a family of varieties $\pi: Y \to X$, one would like to understand how many varieties in the
family have a rational point. 
To this end, we study the following counting function
$$N(U,H,\pi,B) = \#\{ x \in U(k): H(x) \leq B, ~x \in \pi(Y(k)) \},$$
for suitable open subsets $U\subset X$.
As discovered in \cite{Lou13} and \cite{LS16}, the asymptotic behaviour of such counting functions
is controlled by the Galois action on the irreducible components of fibres over the codimension $1$ points of $X$.
We work with the following types of fibres, first defined in \cite{LSS17}.
\begin{definition}
	Let $x \in X$ with residue field $\kappa(x)$.
	We say that a fibre $Y_x=\pi^{-1}(x)$ is \emph{pseudo-split} if every element
	of $\Gal(\overline{\kappa(x)} / \kappa(x))$ fixes some multiplicity one 
	irreducible component of $Y_x \otimes \overline{\kappa(x)}$.
\end{definition}
Note that if $Y_x$ is \emph{split}, i.e.~contains a multiplicity one irreducible component which is geometrically
irreducible \cite[Def.~0.1]{Sko96}, then $Y_x$ is pseudo-split.

The large sieve was employed in \cite{LS16} to give upper bounds for
$N(\PP^n,H,\pi,B)$.
Good upper bounds are not realistic in our generality,
but we are able to obtain the following zero density result, which generalises  \cite[Thm.~1.1]{LS16}.

\begin{theorem} \label{thm:equi_pseudo_split}
	Let $X$ be an almost Fano variety over $k$ and $H$ an anticanonical height function on $X$.
	Assume that the rational points are equidistributed on some dense open subset $U \subset X$.
	Let $\pi:Y \to X$ be a proper dominant morphism with $Y$ geometrically integral and non-singular.
	Assume that there is a 
	non-pseudo-split fibre over some codimension one point of $X$. Then
	$$\lim_{B \to \infty} \frac{N(U,H,\pi,B)}{N(U,H,B)} = 0.$$	
\end{theorem}

\subsubsection{Friable integral points}
Friable numbers are a fundamental  tool in analytic number theory.
A comprehensive survey on what is known about their distribution 
can be found in  \cite{ten}. We 
introduce the following notion of friable integral points.
(Note that, as we are working in a geometric setting,
it is preferable to use
the term ``friable'' over  ``smooth''.)

\begin{definition}
	Let $X$ be a finite type scheme over $\oo_k$ and $Z \subset X$ a closed subscheme.
	For $y>0$, we say that an integral point $x \in X(\oo_k)$ is {\em $y$-friable with respect to $Z$}
	if all non-zero prime ideals $\fp\subset \fo_k$ with $ x \bmod \fp \in Z$ satisfy $\Norm \fp \leq y$.
\end{definition}

One recovers the usual notion of a $y$-friable number by taking $X = \mathbb{A}^1_{\ZZ}$ and $Z$
to be the origin. Allowing different subschemes $Z$ is also very natural:
given a polynomial $f \in \ZZ[x]$, a $y$-friable integral point of $\mathbb{A}^1_{\ZZ}$
with respect to the subscheme $Z= \{ f(x) = 0\}$ is an integer $x$ such that   $f(x)$ is $y$-friable. Lagarias and Soundararajan \cite[Thm.~1.4]{sound} have investigated the case of the linear equation $X=\{x_1+x_2=x_3\}\subset \AA_\ZZ^3$, with $Z=\{x_1x_2x_3=0\}$. Given $c>8$, they assume GRH and succeed in proving that there are infinitely many (primitive) $y$-friable integral points with respect to $Z$, for any 
$y> (\log \max\{|x_1|,|x_2|,|x_3|\})^{c}$. (An unconditional version of this result is available in recent work of Harper \cite{harper}, for $c$ large enough.)
We can give the following  zero density result in our setting.

\begin{theorem} \label{thm:equi_friable}
	Let $X$ be an almost Fano variety over $k$ and $H$ an anticanonical height function on 
	$X$. Assume that the rational points are equidistributed on some dense open subset 
	$U \subset X$. Let $y$ be fixed and $Z \subset X$ a divisor. 
	Let $\mathcal{X}$ be a model of $X$ over $\fo_k$ and $\mathcal{Z}$ 
	the closure of $Z$ in $\mathcal{X}$. Then 
	$$\lim_{B \to \infty} \frac{\#\{x \in U(k): H(x) \leq B, ~x \text{ is $y$-friable with respect to } \mathcal{Z} \}}{N(U,H,B)} = 0.$$	
\end{theorem}

Here,  a \emph{model} is a proper scheme $\mathcal{X} \to \Spec \fo_k$
whose generic fibre is isomorphic to $X$. A model can be obtained, for example, by choosing an embedding $X \subset \PP^n_k$ and letting $\mathcal{X}$ be the closure of $X$ inside $\PP^n_{\fo_k}$. 
Note that in 
Theorem \ref{thm:equi_friable}
 we view $x \in U(k) \subset \mathcal{X}(\fo_k)$, so that its reduction $x \bmod \fp \in \mathcal{X}(\FF_\fp)$ is well-defined.
 
It is crucial in Theorem \ref{thm:equi_friable} that $Z$ be a divisor; the conclusion can fail for higher codimension subvarieties. 
For an example of this phenomenon in the affine setting, consider  $\mathcal{X} = \mathbb{A}^2_\ZZ$ and $\mathcal{Z}$ the origin. Then a $y$-friable integral point with respect to $\mathcal{Z}$ is a pair
of integers $(x_1,x_2)$ whose greatest common divisor is $y$-friable; clearly a positive proportion
of all pairs of integers satisfy this property.

\subsection{Sieving on quadrics}

In many cases it is possible to get 
quantitatively stronger versions of the previous results. We pursue this for smooth quadric hypersurfaces, but we
expect that results of a similar flavour go through for hypersurfaces of arbitrary  degree.  The advantage of  quadrics 
is that sharper bounds are available through the  smooth $\delta$-function variant of the Hardy--Littlewood circle method. 
Note that smooth quadric hypersurfaces are flag varieties; 
hence they are Fano and have equidistribed rational points \cite[\S6.2.4]{Pey95}.

For the remainder of this section $X \subset \PP_\QQ^n$ 
is a smooth quadric hypersurface of dimension at least $3$ over $\QQ$
and  $H:\PP^{n}(\QQ)\to \RR$ is the standard exponential height function associated to the supremum norm. There is a natural choice of model $\mathcal{X}$ given by the closure of $X$ inside $\PP_\ZZ^n$; we shall abuse notation and write $X(\ZZ) = \mathcal{X}(\ZZ)$ and 
$X(\ZZ/m\ZZ) = \mathcal{X}(\ZZ/m\ZZ)$.

\subsubsection{A version of the Selberg sieve}

Our fundamental tool will be a version of the Selberg sieve for rational points on 
quadrics.
Let $m\in \NN$ be fixed once and for all.
For each prime  $p$ we suppose that we are given a 
non-empty  set of residue classes $\Omega_{p^m}\subseteq X(\ZZ/p^m\ZZ)$. Our goal is to measure the density of points  $x\in X(\QQ)$ whose reduction modulo $p^m$ lands in $\Omega_{p^m}$ for each prime $p$. Namely, we are interested in the behaviour of the  counting function
$$
N(X,H,\Omega, B)=
\# \{ x \in X(k) : H(x) \leq B, ~x \bmod p^m \in \Omega_{p^m} \mbox{ for all } p\}
$$
as $B\to \infty$, where $\Omega = (\Omega_{p^m})_p$.
This
has order of magnitude $B^{n-1}$ when  $\Omega_{p^m}=X(\ZZ/p^m\ZZ)$ for all $p$, but we expect it to be significantly smaller when 
$\Omega_{p^m} $ is a proper subset of  $X(\ZZ/p^m\ZZ)$ for many primes $p$.
We define the  density function
\begin{equation}\label{eq:omega}
\omega_p=1-\frac{\# \Omega_{p^m}}{\#X(\ZZ/p^m\ZZ)} \in [0,1],
\end{equation} 
for any prime $p$. 
The following is our main result for quadrics.

 \begin{theorem}\label{thm:Selberg}
 Assume that $X\subset \PP^{n}$ is a smooth quadric of dimension at least $3$ over $\QQ.$
Let $m\in \NN$ and 
let $\Omega_{p^m}\subseteq X(\ZZ/p^m\ZZ)$  
for each prime $p$. Assume that 
$$
0\leq  \omega_p<1, \quad \text{ for all $p$}.
$$
Then,
for any $\xi\geq 1$ and any $\ve>0$, we have
 $$
 N(X,H,\Omega, B)
\ll_{\ve,X}
\frac{B^{n-1}}{G(\xi)}
+\xi^{m(n+1)+2+\ve} B^{(n+1)/2+\ve},
$$
where
$
G(\xi)=\sum_{\substack{k<\xi}}
\mu^2(k)\prod_{p\mid k} 
\left(\frac{\omega_p}{1-\omega_p}\right).
$
 \end{theorem}

The implied constant in this upper bound is allowed to depend on the choice of $\ve$ and the quadric $X$. 
In order to prove  Theorem \ref{thm:Selberg}
we shall use Heath-Brown's version 
\cite{HB} of the  circle method
to study the distribution of zeros of isotropic quadratic forms
that are constrained to lie in a fixed set of congruence classes. The main result, 
Theorem \ref{t:circle},  is uniform in the modulus and may be of independent interest. Once combined with the Selberg sieve, 
it easily leads to the statement of Theorem \ref{thm:Selberg}. 
In fact, although  not  the focus of our present investigation,  
Theorem \ref{t:circle} gives an effective strong approximation result which could 
also be fed into  lower bound sieves, in the spirit of 
 work by Nevo and Sarnak \cite{NS} on the affine linear sieve for homogeneous spaces. 
Finally, by appealing to work of Browning and Vishe \cite{BV} instead of \cite{HB}, we remark that it would be possible 
to obtain a version of Theorem \ref{thm:Selberg} over arbitrary number fields, and to extend the results in the next section to a similar level of generality.

\subsubsection{Applications} 
We now give some applications of Theorem \ref{thm:Selberg}, 
which serve to strengthen  the results in  \S\ref{sec:equi_intro} 
for  smooth quadrics $X\subset \PP^n$ of dimension at least $3$ which are defined over $\QQ$.

To begin with, an old result of 
Cohen  \cite{cohen} and Serre  
 \cite[Thm.~13.3]{Ser97} gives a quantitative improvement of Theorem \ref{thm:equi_thin} when $X$ is projective space.  
The following result extends this to quadrics. 
 
\begin{theorem} \label{thm:quadrics_thin}
	Let $\Upsilon \subset X(\QQ)$ be a thin set. Then there exists $\delta_n > 0$ such that
	$$
	\#\{x\in \Upsilon: H(x)\leq B\} \ll_{\Upsilon,X} B^{n-1-\delta_n}.$$
\end{theorem}

We shall see in 
\S \ref{s:thin} that any $\delta_n<\frac{1}{2}-\frac{7}{2(n+4)}$ is admissible. In particular,  $\delta_n$ approaches $\frac{1}{2}$ as $n\to \infty$, which is the  saving recorded in \cite[Thm.~13.3]{Ser97}.
A well-known application of the latter result is that almost all integer polynomials $f$ of degree $n$
have Galois group the symmetric group $S_n$. (Here we define $\Gal(f)$ to be the Galois group of the splitting field of $f$ over $\QQ$.) Theorem~\ref{thm:quadrics_thin} yields a similar application, but where the coefficients run over a thinner set.

\begin{example}
	Let $n \geq 4$. We claim that
	$$
	\#\left\{
	f(x) = a_nx^n + \dots + a_0 \in \ZZ[x] : 
	\begin{array}{ll}
		|a_i| \leq B, \Gal(f) \neq S_n, \\
		2a_n^2 = a_{n-1}^2 + \dots + a_0^2 
	\end{array}
	\right\}  \ll B^{n-1-\delta_n}.
	$$
	To see this, note that the polynomial $x^n - x^{n-1} - 1$ lives in this family and is irreducible with Galois group $S_n$, by
 the remarks at the end of \cite[\S4.4]{Ser08}.
	This implies that the generic Galois group in the family is also $S_n$. Hilbert's irreducibility theorem \cite[Thm.~3.3.1]{Ser08}
	now implies that the Galois group becomes strictly smaller only on some thin subset of the set of rational points on the associated quadric
	hypersurface. The claim now follows easily from Theorem~\ref{thm:quadrics_thin}.
\end{example}

Our next result concerns fibrations. We extend  
 \cite[Thm.~1.2]{LS16}, 
in which  the base is $\PP^n$, to a result involving quadric hypersurfaces.  
To state the result, we recall the definition of the $\Delta$-invariants from 
\cite{LS16}.
Let $\pi:Y \to Z$ be a dominant map of non-singular proper varieties
over a number field $k$ with geometrically integral
generic fibre.
 For each codimension $1$ point $D \in Z^{(1)}$, the
absolute Galois group $\Gal(\overline{\kappa(D)}/\kappa(D))$ of the residue field of $D$
acts on the irreducible components of $\pi^{-1}(D) \otimes \overline{\kappa(D)}$;
we choose a finite subgroup $\Gamma_D(\pi)$ through which
this action factors. 
As in \cite[Eq.~(1.4)]{LS16}, we then define
$\delta_D(\pi)=\#\Gamma_D^\circ(\pi)/\#\Gamma_D(\pi)$, where
$\Gamma_D^\circ(\pi)$ is the set of 
$\gamma \in \Gamma_D(\pi)$ which fix some multiplicity 1  irreducible component of  $\pi^{-1}(D) \otimes \overline{\kappa(D)}$.
Let
\begin{equation} \label{def:Delta}
	\Delta(\pi) = \sum_{D \in Z^{(1)}}( 1 - \delta_D(\pi) ).
\end{equation}
For the next two results we recall our assumption that 
 $X\subset \PP^n$ is a
   smooth quadric of dimension at least $3$ which is defined over $\QQ$.
We shall deduce the following 
result from Theorem \ref{thm:Selberg}.

\begin{theorem} \label{thm:quadrics_Delta}
	Let $\pi:Y \to X$ be a dominant proper map with geometrically integral
	generic fibre and $Y$ non-singular. Then
	$$
	N(X,H,\pi,B) \ll_{X,\pi} \frac{B^{n-1}}{(\log B)^{\Delta(\pi)}}.
	$$	
\end{theorem}

Note that 
$\Delta(\pi)>0$ if and only if there is 
a non-pseudo-split fibre over some $D\in X^{(1)}$. 
Thus Theorem \ref{thm:quadrics_Delta}  is a refinement of 
Theorem 
\ref{thm:equi_pseudo_split} in the special case that $X$ is a quadric
and $k=\QQ$. 
As in \cite[Conj.~1.6]{LS16}, we expect Theorem~\ref{thm:quadrics_Delta} to be sharp for the related
problem of counting everywhere locally soluble fibres, provided  there is an everywhere locally soluble
fibre and the fibre over every codimension $1$ point contains an irreducible component of multiplicity $1$.
As outlined in the setting of
fibrations over $\PP^n$ \cite[\S 5]{LS16}, 
Theorem~\ref{thm:quadrics_Delta} has several applications.  For example,
using a variant of the proof of \cite[Thm.~5.10]{LS16},
one can obtain a version for quadrics of
Serre's result  \cite[Thm.~2]{Ser90} on
zero loci of Brauer group elements.

Our final application of Theorem \ref{thm:Selberg} refines Theorem \ref{thm:equi_friable} for 
quadrics over $\QQ$. 

\begin{theorem} \label{thm:quadrics_friable}
	Let $y>0$ and let  $Z \subset X$ be a divisor. Then 
	$$\#\{x \in X(\QQ): H(x) \leq B, ~x \text{ is $y$-friable with respect to } Z \} 
	\ll_{X,y,Z} \frac{B^{n-1}}{(\log B)^{r(Z)}},$$
	where $r(Z)$ is the  number of irreducible components of $Z$.	
	\end{theorem}

\subsection*{Layout of the paper}
In \S \ref{sec:Hensel} we collect together  some versions of Hensel's lemma which will be required
in our proofs. In \S\ref{sec:equi_sieves} we prove the results stated in \S\ref{sec:equi_intro},
and also obtain some general volume estimates which will be required for our results
concerning quadrics.
Theorem~\ref{thm:Selberg} will be proved in \S\ref{s:modM} and \S \ref{s:weights}.  
Finally, 
 Theorems \ref{thm:quadrics_thin}, \ref{thm:quadrics_Delta} and \ref{thm:quadrics_friable}  
 will be deduced in \S \S\ref{s:thin}--\ref{s:fry}.

\subsection*{Notation}
For a smooth variety $X$ over a field $k$, let  $\Br X = \HH^2(X, \Gm)$ denote its Brauer group and $\Br_0 X = \im(\Br k \to \Br X)$.
Let $\fp$ be a non-zero prime ideal of the ring of integers $\fo_k$ of a number field $k$.
We let $\FF_\fp$
be the residue field of $\fp$, let $\Norm \fp = \#\FF_\fp$ be its norm, and let $\fo_\fp$
be the completion of $\fo_k$ at $\fp$.
For a variety $X$ over a field $k$ and an extension $k \subset L$, we let $X \otimes L = X \times_{\Spec k} \Spec L$.

\subsection*{Acknowledgements}
The authors are very grateful to 
Hung Bui, Christopher Frei,  Adam Harper, Roger Heath-Brown, and Damaris Schindler 
for useful conversations. 
Thanks are also due to the anonymous referee for simplifying the deduction of Theorem \ref{thm:Selberg} from Theorem \ref{t:circle}.
During the preparation of this  paper the first author was
supported by 
EPSRC grant EP/P026710/1
and  by the NSF under Grant No.\ DMS-1440140,  while  in residence at 
the {\em 
Mathematical Sciences Research Institute} in Berkeley, California,
during the Spring 2017 semester.

\section{Hensel's lemma and transversality} \label{sec:Hensel}
We begin with some versions of Hensel's lemma.
Throughout this section $k$ is a number field and $\fp$ is a non-zero prime ideal of  $\fo_k$.

\subsection{A quantitative version of Hensel's lemma}

Versions of the following lemma
have been known for some time.

\begin{lemma} \label{lem:Hensel}
	Let $X \to \Spec \oo_\fp$ be a smooth finite type morphism of relative dimension $n$.
	Let	$x_0 \in X(\FF_\fp)$ and $m\in \NN$. Then 
	$$\#\{x \in X(\oo_\fp/\fp^m) : x \bmod \fp = x_0 \} = (\Norm \fp)^{n(m-1)}.$$
	In particular $\# X(\oo_\fp/\fp^m) = \#X(\FF_\fp) (\Norm \fp)^{n(m-1)}.$
\end{lemma}
\begin{proof}
	We want to calculate
	the number of morphisms $\Spec \oo_\fp/\fp^m \to X$ whose image is $x_0$.
	This is in bijection with the set of local $\oo_\fp$-algebra homomorphisms
	$\Hom(\OO_{X, x_0},\oo_\fp/\fp^m).$
	Since $\oo_\fp/\fp^m$ is Artinian it is complete. Hence by the universal
	property of the completion we find that
	$$\Hom(\OO_{X,x_0},\oo_\fp/\fp^m) \cong \Hom(\widehat{\OO}_{X,x_0},\oo_\fp/\fp^m).$$
	However, as $X \to \Spec \oo_\fp$ is smooth, we have 
	$\widehat{\OO}_{X,x_0} \cong \oo_\fp[[t_1,\ldots,t_n]]$
	as local $\oo_\fp$-algebras by \cite[Ex.~6.2.2.1]{Liu02}.
	To prove the result, it suffices to note that 
$$
		\#\Hom(\oo_\fp[[t_1,\ldots,t_n]], \oo_\fp/\fp^m) = (\Norm \fp)^{n(m-1)}.
$$
	Indeed, every element of $\Hom(\oo_\fp[[t_1,\ldots,t_n]], \oo_\fp/\fp^m)$
	has the form
	$$t_i \mapsto a_i, \quad i \in \{1,\ldots,n\},$$
	for non-units $a_1,\ldots,a_n \in \oo_\fp/\fp^m$. But 
	$\oo_\fp/\fp^m$ has exactly $(\Norm \fp)^{m-1}$ non-units.
\end{proof}

\subsection{Transverse intersections}

Following Harari \cite[\S 2.4.2]{Harari}, we use the following notion of intersection multiplicity.

\begin{definition} \label{def:trans}
	Let $X \to \Spec \oo_\fp$ be a smooth finite type morphism of relative dimension $n$ and
	let $D \subset X$ be an irreducible divisor which is flat over $\fo_\fp$.
	Let $x \in X(\oo_\fp)$ 
	be such that $x \notin D$ and let $t = 0$ be a local equation for $D \subset X$ 
	on some affine patch $U \subset X$ containing $x$.
	We define the {\em intersection multiplicity} of  $x$ and $D$ above $\fp$
	to be the integer $\iota$ which satisfies
	$$x^*t = \varpi^\iota,$$
	where $\varpi$ denotes a uniformising parameter of $ \oo_\fp$ and $x^*t$
	is the pull-back of $t$ via $x: \Spec \oo_\fp \to U$.
	We say that $x$ and $D$ meet \emph{transversely} above $\fp$ if $\iota=1$.
\end{definition}
This definition is independent of the choice of $t$ and $\varpi$.
Moreover, whether or not $x$ and $D$ meet transversely above $\fp$ 
only depends on $x \bmod \fp^2$. 
In particular, asking whether a point in $X(\oo_\fp/\fp^2)$ meets $D$ transversely
above $\fp$ is well-defined.

\begin{proposition} \label{prop:transverse}
	Let $X \to \Spec \oo_\fp$ be a smooth finite type morphism of relative dimension $n$. 
	Let $D \subset X$ be a flat irreducible divisor and let $x_0 \in D(\FF_\fp)$ be a smooth
	point of $D$. Then
	$$\left| \, \#\left\{x \in X(\oo_\fp/\fp^2):
	\begin{array}{l}
		x \bmod \fp = x_0, \\
		x \text{ meets $D$ transversely}
	\end{array} \right\} - (\Norm \fp)^{n} \, \right| \leq (\Norm \fp)^{n-1}.$$
\end{proposition}
\begin{proof}
	The problem is local around $x_0$. Thus without loss of generality,
	we may assume that $X$ is affine and that $D$ is smooth and has the equation $t=0$.
	Let $T(x_0)$ be the cardinality in question.
	Lemma \ref{lem:Hensel} shows that $T(x_0) \leq (\Norm \fp)^{n}$.

	For the reverse inequality, we use an argument inspired by 
	the proof on p.~233 of \cite{Harari}. 
	Let $\varpi \in \oo_\fp$ be a uniformising parameter and let $U_\fp \subset \oo_\fp^*$
	be a collection of $(\Norm \fp - 1)$ units which are distinct modulo $\fp$.
	For $u \in U_\fp$ consider the divisors
	$$D_u: \quad t = u \varpi \quad \subset X.$$
	A simple calculation shows that each $D_u$ is also smooth.
	Moreover, we have $D_u \otimes \FF_\fp = D \otimes \FF_\fp$ and
	$$D_u \cap D = D_{u'} \cap D_u = D \otimes \FF_\fp,$$
	 for  $u'  \neq u$.
	Clearly any point $x \in D_u(\oo_\fp/\fp^2)$ with
	$x \bmod \fp = x_0$ meets $D$ transversely at $x_0$.
	Applying Lemma \ref{lem:Hensel} to the $D_u$, we therefore deduce that
	\begin{align*}
		T(x_0)  \geq \sum_{u \in U_\fp} 
		\#\left\{x \in D_u(\oo_\fp/\fp^2): 	x \bmod \fp = x_0 \right\} 
		& = \sum_{u \in U_\fp} (\Norm \fp)^{n-1} \\&= (\Norm \fp - 1)(\Norm \fp)^{n-1}. \qedhere
	\end{align*}
\end{proof}

From Proposition \ref{prop:transverse} we easily deduce the following global statement.

\begin{corollary} \label{cor:transverse}
	Let $X \to \Spec \oo_k$ be a smooth  finite type morphism of relative dimension $n$. 
	Let $D \subset X$ be a flat irreducible divisor and $Z \subset D$ a 
	closed subscheme which contains the non-smooth locus of $D$ 
	and is of codimension at least $2$ in $X$.
	Then
	$$\#\left\{x \in X(\oo_\fp/\fp^2):
	\begin{array}{l}
		x \bmod \fp \in (D \setminus Z)(\FF_\fp), \\
		x \text{ meets $D$ transversely}
	\end{array} \right\}
	\hspace{-0.1cm}
	= 
	\hspace{-0.1cm}
	\#D(\FF_\fp)(\Norm \fp)^{n} + O((\Norm \fp)^{2(n-1)}),$$
	where the implied constant depends on $Z$ and $D$.
\end{corollary}
\begin{proof}
	Applying Proposition \ref{prop:transverse} and the Lang--Weil estimates \cite{LW54},
	we find that the cardinality in question equals
	\[
		\#(D \setminus Z)(\FF_\fp)\left((\Norm \fp)^n + O((\Norm \fp)^{n-1}) \right) 
		 = \#D(\FF_\fp)(\Norm \fp)^n + O((\Norm \fp)^{2(n-1)}).  \qedhere
	\]
\end{proof}

\begin{remark}
	Proposition \ref{prop:transverse} is a quantitative improvement
	of the fact, often used in proofs,
	that any smooth $\FF_\fp$-point on $D$ lifts
	to an $\oo_\fp$-point of $X$ which meets $D$ transversely above $\fp$. 
	(cf.~the proof of Theorem 2.1.1 on p.~233 of \cite{Harari} or the proof
	of \cite[Thm.~4.2]{LSS17}).
\end{remark}

\section{Equidistribution and sieving} \label{sec:equi_sieves}
In this section we prove the results stated in \S \ref{sec:equi_intro}.

\subsection{Tamagawa measures and equidistribution} \label{sec:Tamagawa}
We first recall some notions and results concerning Tamagawa measures and
equidistribution of rational points for varieties over a number field $k$. Our references here are \cite{Pey95} and \cite{CLT10}.

\subsubsection{Tamagawa measures}
We now recall the construction of Peyre's Tamagawa measure. 
(In practice we will use Lemma \ref{lem:Salberger} for calculations.)
Choose Haar measures $\mathrm{d}x_v$ on each $k_v$ such that $\vol(\oo_v)=1$ for all but finitely
many $v$. These give rise to a Haar measure $\mathrm{d}x$ on the ad\`{e}les $\Adele_k$ of $k$; we normalise
our Haar measures so that $\vol(\Adele_k/k)=1$ with respect to the induced quotient measure.

Now let $X$ be a smooth projective variety over $k$ and let $\| \cdot \| = (\| \cdot \|_v)_{v \in \Val(k)}$ be a choice of adelic metric on the canonical bundle of $X$ as in \cite[\S\S 2.1--2.2]{CLT10}.
Let $\omega$ be a top degree differential form on some dense open subset $U \subset X$. By a classical construction 
\cite[\S2.1.7]{CLT10}, for any place $v$ of $k$ we obtain a measure $|\omega|_v$ on $U(k_v)$ 
which depends on the choice of $\mathrm{d}x_v$. The measure $|\omega|_v/\|\omega\|_v$ turns out to be 
independent of $\omega$. 
Peyre's local Tamagawa measure 
$\tau_{\|\cdot\|_v}$ on $X(k_v)$ is obtained by glueing these measures.
The product of the $\tau_{\|\cdot\|_v}$ does not converge in general, to which end convergence factors are introduced. 
Let $M_X$ be the free part of the geometric N\'{e}ron--Severi group $\NS(\bar{X})$, with Artin $L$-function $L(s,M_X) = \prod_{v} L_v(s,M_X)$.  (For an archimedean place $v$ we set $L_v(s,M_X) = 1$.)
Under the additional assumption that $\HH^1(X, \OO_X) = \HH^2(X,\OO_X) = 0$, it is proved in \cite[Thm.~1.1.1]{CLT10}
that $\lambda_v = L_v(1,M_X)^{-1}$ are a collection of convergence factors. In this way 
$$
\tau_{\|\cdot\|} = \lim_{s \to 1} (s-1)^{\rho(X)} L(s,M_X) \prod_{v \in \Val(k)} \lambda_v \tau_{\|\cdot\|_v} $$
yields a  measure on $X(\Adele_k)$, called
Peyre's global Tamagawa measure.

The above construction  applies when $X$ is almost Fano. 
In this case we also denote the measure by $\tau_H$,
where $H$ is the anticanonical height function associated to the adelic metric $\|\cdot \|$.
The conjecture for the leading constant
in \eqref{conj:Manin} is
$
	c_{X,H} = \alpha(X) \beta(X) \tau_H(X(\Adele_k)^{\Br}),
$
where $X(\Adele_k)^{\Br}$ is the subset of $X(\Adele_k)$ which is orthogonal to $\Br X$, $\beta(X) = \#\HH^1(k, \Pic \bar{X})$,
and $\alpha(X)$ is Peyre's ``effective cone constant''. (The precise definition of
$\alpha(X)$, which
can be found in \cite[Def.~2.4]{Pey95}, is irrelevant here.)

The following result implies that the Tamagawa measure $\tau_H(X(\Adele_k)^{\Br})$
is essentially given by a product of local densities.

\begin{lemma} \label{lem:WWA}
	Let $X$ be an almost Fano variety over $k$.
	Then $\Br X/ \Br_0 X$ is finite and there exists a finite set of places $S$ and a compact open 
	subset $A \subset \prod_{v \in S} X(k_v)$ such that
	$X(\Adele_k)^{\Br} = A \times \prod_{v \notin S} X(k_v).$
\end{lemma}
\begin{proof}
	The finiteness of $\Br X/ \Br_0 X$ is \cite[Lem.~6.10]{Sal98}. 
	For each $b \in \Br X$, the map 
	$X(\Adele_k) \to \QQ/\ZZ$
	induced by the Brauer pairing
	 is locally constant \cite[Cor.~6.7]{Sal98}. Thus the inverse
	image of $0$ is a compact open subset. As the Brauer pairing factorises through the finite group
	$\Br X/ \Br_0 X$, the result follows.
\end{proof}

We calculate the Tamagawa measure using the following formula,
which follows immediately from \cite[Thm.~2.14(b)]{Sal98} (cf.~\cite[Cor.~2.15]{Sal98}).

\begin{lemma} \label{lem:Salberger}
	Let $X$ be a smooth projective variety of dimension $n$ over $k$ 
	with a choice of adelic metric $\| \cdot \|$ on $-K_X$.
	Let $\mathcal{X}$ be a model of $X$ over $\fo_k$.
	Then there exists a finite set $S$ of prime ideals of $\fo_k$ 
	such that 
	$$\tau_{\|\cdot\|_\fp}(\{x \in X(k_\fp): x \bmod \fp^m \in \Omega\}) = \frac{\#\Omega}{(\Norm\fp)^{mn}},$$
		for any $\fp \notin S$, any $m >0$ and any $\Omega \subset \mathcal{X}(\oo_\fp / \fp^m )$. 
\end{lemma}

\subsubsection{Equidistribution}
We now recall the definition of equidistribution of rational points, as given by Peyre  \cite[\S 3]{Pey95} and further developed in \cite[\S2.5]{CLT10}.

\begin{definition} \label{def:equi}
	Let $X$ be an almost Fano variety over a number field $k$ with $X(k) \neq \emptyset$.
	Let $H$ be an anticanonical height function on $X$ with associated Tamagawa measure
	$\tau_H$. We say that the rational points on $X$ 
	are {\em equidistributed} with respect to $H$ and some dense open subset $U \subset X$ 
	if $U(k) \neq \emptyset$ and for any open subset $W \subset X(\Adele_k)$ with $\tau_H(\partial W) = 0$,  we have
	\begin{equation} \label{def:equi_formula}
		\lim_{B \to \infty} \frac{\#\{ x \in U(k) \cap W: H(x) \leq B \}}{\#\{ x \in U(k): H(x) \leq B \}}
		= \frac{\tau_H(W \cap X(\Adele_k)^{\Br})}{\tau_H(X(\Adele_k)^{\Br})}.
	\end{equation}
\end{definition}

As proved in \cite[\S 3]{Pey95}, if the equidistribution property holds with respect to \emph{some}
choice of anticanonical height, then it holds for \emph{all} choices of anticanonical height. Moreover,
the equidistribution property holds if one knows
\eqref{conj:Manin} with Peyre's constant with respect to all choices of adelic 
metric on the anticanonical bundle.  (In fact, it follows from \cite[Prop.~2.5.1]{CLT10} and 
the Stone--Weierstrass theorem that one need only prove this with respect to all \emph{smooth} adelic metrics.)
	
\begin{example} \label{ex:prod_measures}
	Assume that the rational points on $X$ are equidistributed
	with respect to $H$ on a dense open subset $U \subset X$.
	Let $\mathcal{X}$ be a model of $X$ over $\fo_k$ and 
	let $S$ be a finite set of non-zero primes ideals of $k$.  
	Let $m >0$
	and $\Omega_{\fp^m} \subset \mathcal{X}(\oo_\fp / \fp^m )$ for $\fp \in S$.
	Then Lemma~\ref{lem:WWA}, Lemma \ref{lem:Salberger} and \eqref{def:equi_formula}
	imply that
	\begin{equation*}
		\lim_{B \to \infty} 
		\frac{\#\{ x \in U(k): H(x) \leq B, \,x \bmod \fp^m \in \Omega_{\fp^m} \, \forall \fp \in S \}}
		{N(U,H,B)}
		\ll \prod_{\fp \in S} \frac{ \#\Omega_{\fp^m}}{\#\mathcal{X}(\oo_\fp/\fp^m)},
	\end{equation*}	
	where the implied constant
depends on $\mathcal{X},H,m$ but 
is \emph{independent} of $S$ and  $\Omega_{\fp^m}$.
\end{example} 

\begin{remark}
	The equidistribution property can be viewed as a quantitative version of weak approximation;
	indeed, if $W$ is an open neighbourhood of a point $(x_v) \in X(\Adele_k)^{\Br}$ with $\tau_H(\partial W) = 0$,
	then \eqref{def:equi_formula} implies that $W$ contains many rational points. In particular
	$\overline{X(k)} = X(\Adele_k)^{\Br}$ and so the Brauer--Manin obstruction is the only obstruction to weak approximation.
	Moreover, weak approximation holds on $X$ away from the finite set of places $S$ by Lemma \ref{lem:WWA}.
\end{remark}

\begin{remark}
	A natural problem is to formulate a 
	version of Definition \ref{def:equi} for the ``thin'' version of Manin's conjecture. 
	Here one should replace the condition $x \in U(k)$ from the counting functions in \eqref{def:equi_formula}
	by the condition
	that $x$ lies in the complement of an appropriate thin subset of $X(k)$. It would be interesting to see whether this version holds for the examples considered by Le Rudulier in \cite{LeR13}.
\end{remark}
	
\subsection{Thin sets} \label{sec:thin}
We recall Serre's definition of thin sets from \cite[\S3.1]{Ser08}. 

\begin{definition}
Let $X$ be an integral variety over a field
$F$. 
A {\em type $I$ thin subset}  is a set of the form $Z(F) \subset X(F)$, where $Z$ is a
closed subvariety with $Z \neq X$. A {\em type $II$ thin subset} is a set of the form
$\pi(Y(F))$, where $\pi : Y \to X$ is a generically finite dominant morphism with $\deg \pi \geq 2$ and $Y$ geometrically integral. A {\em thin subset} is a subset  contained in a finite
union of thin subsets of type $I$ and $II$.
\end{definition}

To prove Theorem \ref{thm:equi_thin}, we require  information on  thin sets modulo $\fp$.

\begin{lemma} \label{prop:thin}
	Let $k$ be a number field, let
	$X \to \Spec \oo_k$ be a smooth integral finite type scheme of relative dimension $n$
	and $\Upsilon \subset X(\oo_k)$ be thin in $X(k)$.
	\begin{enumerate}
		\item If $\Upsilon$ has type $I$ then 
		$\#(\Upsilon  \bmod \fp) \ll_{\Upsilon} (\Norm \fp)^{n-1}$.
		\item If $\Upsilon$ has type $II$, then there exists a finite Galois extension $k_\Upsilon/k$ and a constant $c_\Upsilon \in (0,1)$
		such that for all primes $\fp$ of $\fo_k$ which  split completely in $k_\Upsilon$ we have
		$ \#(\Upsilon \bmod \fp) \leq c_\Upsilon (\Norm \fp)^n + O_{\Upsilon}( (\Norm \fp)^{n-1/2})$.
	\end{enumerate}
\end{lemma}
\begin{proof}
	The first part  follows from applying the Lang--Weil
	estimates \cite{LW54}
	to each component of the closure of $\Upsilon$.
	The second part is  \cite[Thm.~3.6.2]{Ser08}.
\end{proof}

\subsubsection{Proof of Theorem \ref{thm:equi_thin}}
To prove the theorem we may reduce to the case of thin sets of of type $I$ or $II$. The case
of type $I$ is easy, so we assume that $\Upsilon$ is a thin set of type $II$.
Let $z > 1$ and 
let $\mathcal{P}$ be the set of primes $\fp$ in $\fo_k$ which split completely in $k_\Upsilon$.
  As the rational points on $X$ are equidistributed, 
it follows from Example \ref{ex:prod_measures},  Lemma   \ref{prop:thin}, and the Lang--Weil estimates \cite{LW54}
that
\begin{align*}
	\lim_{B \to \infty} &
\frac{\#\{ x \in U(k): H(x) \leq B, ~x \bmod \fp \in (\Upsilon \bmod{\fp}) \, \forall \Norm \fp \leq z \}}
	{N(U,H,B)}\\
&	\hspace{5cm}\ll_{X,H}
	\prod_{\substack{\fp \in \mathcal{P}\\ \Norm \fp\leq z}}
	\left(
	c_\Upsilon+O_{\Upsilon}\left(\frac{1}{\sqrt{\Norm \fp}}\right)
	\right).
\end{align*}
The set $\mathcal{P}$ is infinite
by the Chebotarev density theorem. Since $0 < c_\Upsilon < 1$, the result  follows on taking $z \to \infty$. \qed

\subsection{Local solubility densities}
Let $k$ be a number field.
We gather some tools for the proof of Theorem \ref{thm:equi_pseudo_split}. This is proved
with
 an analogous strategy to Theorem \ref{thm:equi_thin}, by 
deriving upper bounds for the size of the set in question modulo $\fp^m$, for some $m$. In Lemma  \ref{prop:thin} it was sufficient to take  $m=1$, but 
as first noticed by Serre  \cite{Ser90} (and further  developed in \cite{LS16}),  for fibrations one needs to sieve modulo higher powers of $\fp$.
For example, consider the conic bundle 
\begin{equation} \label{eqn:conic}
	x^2 + y^2 = tz^2 \quad \subset \PP^1_\ZZ \times \AA^1_\ZZ.
\end{equation}
For any odd prime $p$, the fibre over \emph{every} $\FF_p$-point of $\AA_{\ZZ}^1$ has an $\FF_p$-point; but there are clearly
fibres over $\QQ$ which have no $\QQ_p$-point. So
 sieving  modulo $p$ gives no information. One obtains good upper bounds here by sieving modulo $p^2$, using the fact that if $p \equiv 3 \bmod 4$ and the $p$-adic valuation of $t$ is equal to $1$, then the corresponding conic \eqref{eqn:conic} has no $\QQ_p$-point.

These observations were greatly generalised by Loughran and Smeets in \cite{LS16}. The condition that the $p$-adic valuation of $t$ is $1$ can be interpreted geometrically as requiring that a certain intersection is transverse over $p$ (see Definition \ref{def:trans}). The required generalisation is the  following ``sparsity theorem'' from \cite{LS16}, which gives an explicit
criterion for non-solubility at sufficiently large primes.

\begin{proposition} \label{prop:sparsity} 
Let $\pi: Y \to X$ be a dominant morphism of finite type $\oo_k$-schemes with $Y_k$ and $X_k$ smooth geometrically integral $k$-varieties. 
Let $T$ be a reduced divisor in $X$ such that the restriction of $\pi$ to $X\setminus T$ is smooth. 
Then there exists a finite set of prime ideals $S$ and a closed subset $Z\subset T_{\oo_{k,S}}$ containing the singular locus of $T_{\oo_{k,S}}$ and of codimension $2$ in $X_{\oo_{k,S}}$, such that for all non-zero prime ideals $\fp \notin S$ the following holds:

Let $x \in X(\fo_\fp)$ be such that the image of ${x}: \Spec \fo_\fp \to X$ meets $T_{\oo_{k,S}}$ transversally over $\fp$ outside of $Z$ and such that the fibre above ${x} \bmod \fp \in T(\FF_\fp)$ is non-split. Then 
$(Y \times_{X} {x})(\fo_\fp)=\emptyset$; i.e.~the fibre over $x$ has no $\fo_\fp$-point.
\end{proposition}

\begin{proof}
For rational points this is proved in \cite[Thm.~2.8]{LS16}. The adaptation to integral points is straightforward and omitted.
\end{proof}

The  following is the main result of this section. It is phrased in terms of the invariant $\Delta(\pi)$
that was defined in
\eqref{def:Delta}.

\begin{proposition} \label{thm:p^2}
	Let $\pi: Y \to X$ be a dominant morphism of finite type $\oo_k$-schemes 
	with $Y_k$ and $X_k$ smooth geometrically integral $k$-varieties. 
	Assume that the generic fibre of $\pi$ is geometrically integral
	and that $Y(\oo_\fp) \neq \emptyset$ for all primes $\fp$.  
	For any non-zero prime ideal $\fp\subset \fo_k$ let 
	$$
	\Theta_\fp = \#\{ x \in X(\oo_k/\fp^2) : x \notin \pi(Y(\oo_\fp)) \bmod \fp^2\}.
	$$
	Then 
	\begin{align}
			 \label{eqn:upper_bound}
			&\bullet\qquad \frac{\Theta_\fp}{\#X(\oo_k/\fp^2)}
			\ll \frac{1}{\Norm \fp}, \\	
			 \label{eqn:sum_Sp}
			 &\bullet\qquad
			\sum_{\substack{ \Norm \fp \leq B}} 
			\frac{\Theta_\fp \log \Norm \fp}{\#X(\fo_k/\fp^2)}
			 \sim  \Delta(\pi)\log B, \text{ and } \\
		\label{eqn:Euler_product}
		&\bullet\qquad
			\prod_{\substack{ \Norm \fp \leq B}} 			
			\frac{\#(\pi(Y(\oo_\fp)) \bmod \fp^2)}
			{\#X(\oo_k/\fp^2)} \asymp \frac{1}{(\log B)^{\Delta(\pi)}}.
	\end{align}
\end{proposition}
\begin{proof}
	Let $n = \dim X_k$. To begin with
 we claim that 
	\begin{equation} \label{eqn:claim_non_split}
		\Theta_\fp = \#\{x \in X(\FF_\fp) : \pi^{-1}(x) \text{ non-split}\}(\Norm \fp)^{n} + O\left((\Norm \fp)^{2(n-1)}\right).
	\end{equation}
	To prove this, let $T$ be a divisor of $X$ which contains the singular locus of $\pi$
	and let $x \in X(\oo_k/\fp^2)$. If $\pi^{-1}(x \bmod \fp)$ is split then, by 
	the Lang--Weil estimates \cite{LW54} and Hensel's lemma, for large enough  $\fp$  we find that the fibre over $x$ has
	an $\oo_k/\fp^2$-point. Thus for large enough $\fp$ we have
	\begin{equation} \label{eqn:refine}
		\Theta_\fp = \# \left\{ x \in X(\oo_k/\fp^2):
		\begin{array}{l}
			x \notin \pi(Y(\oo_\fp)) \bmod \fp^2, \,x \bmod \fp \in T, \\
			\pi^{-1}(x \bmod \fp) \mbox{ is non-split}
		\end{array} \right\}.
	\end{equation}
	However the Lang--Weil estimates 
	and Lemma \ref{lem:Hensel} yield
	\begin{equation} \label{eqn:LW_p^2}
		\#X(\oo_k/\fp^2) = (\Norm \fp)^n\#X(\FF_\fp) = (\Norm \fp)^{2n} + O\left((\Norm \fp)^{2n-{1/2}}\right)
	\end{equation}
	and 
	$$		\# \{ x \in X(\oo_k/\fp^2): x \bmod \fp \in T\} \ll (\Norm \fp)^{2n-1}.$$
	These and \eqref{eqn:refine} already yield the upper bound  \eqref{eqn:upper_bound}.
	Moreover, Lemma \ref{lem:Hensel} and \eqref{eqn:refine} show that $ \Theta_\fp \leq \#\{x \in X(\FF_\fp) : \pi^{-1}(x) \text{ non-split}\}(\Norm \fp)^{n}$, 
	which gives the upper bound in 
\eqref{eqn:claim_non_split}.
	For the lower bound, let $Z$ be as in Proposition \ref{prop:sparsity}. 
	Then Proposition~\ref{prop:sparsity} and Proposition \ref{prop:transverse} (cf.~the proof
	of Corollary \ref{cor:transverse}) give
	\begin{align*}
		\Theta_\fp & \geq \# \left\{ x \in X(\oo_k/\fp^2):
		\begin{array}{l}
			x \bmod \fp \in T\setminus Z, \, x \text{ meets } 
			T \text{ transversely above }\fp,\\
			\pi^{-1}(x \bmod \fp) \mbox{ is non-split}
		\end{array} \right\} \\
		 		& = \#\{x \in X(\FF_\fp): \pi^{-1}(x) \text{ non-split}\}(\Norm \fp)^{n}
		 + O\left((\Norm \fp)^{2(n-1)})\right),
	\end{align*}
	whence \eqref{eqn:claim_non_split}. 
Here, we have used the fact that 
$$ 
\{x \in T(\FF_\fp): \pi^{-1}(x) \text{ non-split}\}= \{x \in X(\FF_\fp): \pi^{-1}(x) \text{ non-split}\}
$$
for large enough $\fp$.	
	Next, we claim that 
	\begin{equation} \label{eqn:PNT}
		 \sum_{\Norm \fp \leq B} \#\{x \in X(\FF_\fp) : \pi^{-1}(x) \text{ non-split}\} 
	= \frac{\Delta(\pi)B^n}{\log( B^n)}+O\left(\frac{B^n}{(\log B)^2}\right).
	\end{equation}
Indeed, an easy modification of the proof of  \cite[Prop.~3.10]{LS16}, which is stated without an explicit error term, shows that 
$$
		 \sum_{\Norm \fp \leq B} \#\{x \in X(\FF_\fp) : \pi^{-1}(x) \text{ non-split}\} 
=
 \Delta(\pi)
\sum_{\Norm \fp \leq B} (\Norm \fp)^{n-1} +O\left(\frac{B^n}{(\log B)^2}\right),
$$
on using Serre's version of the Chebotarev density theorem \cite[Thm.~9.11]{Ser08}. The claim  \eqref{eqn:PNT}  follows from an application of the prime ideal theorem and partial summation.
	We  obtain \eqref{eqn:sum_Sp} using \eqref{eqn:claim_non_split}, \eqref{eqn:LW_p^2}, 
	\eqref{eqn:PNT} and a further application of partial summation.
Next, taking logarithms it follows from 
\eqref{eqn:upper_bound} that
\begin{align*}
\log \prod_{\substack{ \Norm \fp \leq B}} 			
			\frac{\#(\pi(Y(\oo_\fp)) \bmod \fp^2)}
			{\#X(\oo_k/\fp^2)}
&=			\sum_{\substack{  \Norm \fp \leq B}} \log\left( 1 - \frac{\Theta_\fp}
		{\#X(\fo_k/\fp^2)}\right)  \\
		&= -\sum_{\substack{  \Norm \fp \leq B}} \frac{\Theta_\fp}{\#X(\fo_k/\fp^2)} + O(1).
\end{align*}
On combining this with 
\eqref{eqn:sum_Sp} and partial summation, we deduce that 
$$
\log \prod_{\substack{ \Norm \fp \leq B}} 			
			\frac{\#(\pi(Y(\oo_\fp)) \bmod \fp^2)}
			{\#X(\oo_k/\fp^2)}
		= 
	-\Delta(\pi)\log\log B +O(1).
		$$
The bounds recorded in \eqref{eqn:Euler_product} are now obvious. 
\end{proof}

We give a consequence which is required for the proof of Theorem \ref{thm:quadrics_Delta}. To achieve this we use the  following version of Wirsing's theorem over number fields.

\begin{lemma} \label{lem:Wirsing} 
	Let $g$ be a  non-negative multiplicative arithmetic function on the non-zero ideals of $\oo_k$.
	Assume that there exist $\alpha, \beta > 0$ such that
	\begin{equation} \label{eqn:Wirsing_1}
		\sum_{\Norm \fp \leq x} \frac{g(\fp) \log \Norm \fp }{\Norm \fp}\sim \alpha \log x
	\end{equation}
	as $x \to \infty$ and $g(\fp^v) \leq \beta^v$ for all non-zero 
	prime ideals $\fp$ and all $ v \in \NN$.
	Then there exists $c_{g}>0$ such that 
	$$\sum_{\Norm \fa \leq x} g(\fa) \sim c_{g}
	\frac{x}{\log x}\prod_{\Norm\fp\leq x} \left(1+\frac{g(\fp)}{\Norm\fp} + 
	\frac{g(\fp^2)}{\Norm\fp^2} + \dots \right).$$
\end{lemma}
\begin{proof}
	Over $\QQ$ this is a special case of \cite[Satz~1.1]{wirsing}.
	We deduce the case of a general number field from this as follows.
	Let $g$ be as in the lemma and let $d=[k:\QQ]$. Define the arithmetic function over $\QQ$ via
	$$h(n) = \sum_{\Norm \fa = n} g(\fa).$$
	Note that as ideals of prime norm are prime we have
$$
		h(p)= \sum_{\Norm \fp = p } g(\fp).
$$
	Using unique factorisation of ideals, 
	one easily verifies that $h$ is a non-negative multiplicative function.
	We have $h(p^v) \leq (d\beta)^v$ for all primes $p$ and all $v \in \NN$. Moreover, 
	\begin{align*}
		\sum_{\Norm \fp \leq x} \frac{g(\fp) \log \Norm \fp }{\Norm \fp}
		&= \sum_{p \leq x} \frac{h(p) \log p }{p} +
		\sum_{\substack{p,v \geq 2 \\ p^v \leq x}} \frac{h(p^v) \log p^v }{p^v} 
		=  \sum_{p \leq x} \frac{h(p) \log p }{p} +O(1).
	\end{align*}
Thus $h$ also satisfies the hypotheses of the lemma and it follows that
	$$\sum_{\Norm \fa \leq x} g(\fa) = \sum_{n \leq x} h(n)
	\sim c_{h} \frac{x}{\log x}\prod_{p \leq x}
	\left(1+\frac{h(p)}{p} + \frac{h(p^2)}{p^2} + \dots \right)  .$$
	The asymptotic behaviour of the above product
	is determined by the term $h(p)/p$. 
We deduce that there is a 
constant $c_g' > 0$ such that 
	$$\prod_{p \leq x} \left(1+\frac{h(p)}{p} + \frac{h(p^2)}{p^2} + \dots \right)
	\sim c_g' \prod_{\Norm\fp\leq x} \left(1+\frac{g(\fp)}{\Norm\fp} + 
	\frac{g(\fp^2)}{\Norm\fp^2} + \dots \right) $$
as $x \to \infty,$
since higher order terms and prime ideals of non-prime norm
	do not affect the asymptotic behaviour. This completes the proof.
\end{proof}

Combining Wirsing's result with Proposition \ref{thm:p^2}, we can deduce the 
following.

\begin{corollary} \label{cor:Wirsing}
Assume that $\Delta(\pi)>0$ and that the assumptions of Proposition~\ref{thm:p^2} hold. Let 
	$$ \omega_\fp = 1 - \frac{\#\pi(Y(\oo_\fp) \bmod \fp^2)}{\#X(\oo_k/\fp^2)}, \quad
	G(B)=\sum_{\substack{ \Norm \fa \leq B}} \mu_k^2(\fa )\prod_{\fp\mid \fa}  
	\left(\frac{\omega_\fp}{1-\omega_\fp}\right).
	$$
	where $\mu_k$ is the M\"{o}bius function on the ideals of $\fo_k$. Then
	$$G(B) \asymp (\log B)^{\Delta(\pi)}.$$
\end{corollary}
\begin{proof}
	We shall show that the conditions of Lemma \ref{lem:Wirsing} are satisfied with
	$$g(\fa) = (\Norm \fa)\mu_k^2(\fa) \prod_{\fp \mid \fa} \frac{\omega_\fp}{1 - \omega_\fp}.$$
This function is  non-negative, multiplicative and supported on square-free ideals of $\fo_k$.
Since  $\omega_\fp = O(1/\Norm \fp)$, by 
\eqref{eqn:upper_bound}, we also  have $g(\fp) = O(1)$.
Next, it follows from \eqref{eqn:sum_Sp} that \eqref{eqn:Wirsing_1} holds with $\alpha=\Delta(\pi)$.
Hence
 Lemma~\ref{lem:Wirsing} yields
$$
\sum_{\Norm \fa \leq B} g(\fa)
\sim c\frac{B}{\log B}\prod_{\Norm\fp\leq B} \left(1+\frac{g(\fp)}{\Norm\fp}\right),
$$ 
for a suitable constant $c>0$.
But, in view of \eqref{eqn:Euler_product} we have
$$
\prod_{\Norm\fp\leq B} \left(1+\frac{g(\fp)}{\Norm\fp}\right) =
\prod_{\Norm\fp\leq B} \left(1-\omega_\fp\right)^{-1}
\asymp
(\log B)^{\Delta(\pi)},
$$
Thus 
$$
\sum_{\Norm \fa \leq B} g(\fa)
\asymp B (\log B)^{\Delta(\pi)-1}.
$$
The desired bounds for $G(B)$ now  follow on using  partial summation to remove the factor $\Norm \fa$ in $g(\fa)$.
\end{proof}

\subsection{Proof of Theorem \ref{thm:equi_pseudo_split}}
Let $\pi:Y \to X$ be as in Theorem \ref{thm:equi_pseudo_split}.
First assume that the generic fibre of $\pi$ is not geometrically integral. Then, as $Y$
is smooth over $k$, 
the generic fibre is smooth thus not geometrically connected. Hence we may consider the Stein factorisation
\cite[Cor.~III.11.5]{Har77}
\[\xymatrix{
	X \ar[rr]^\pi \ar[dr]^f & & Y \\ 
	& Z \ar[ur]^g &
	} \] 
of $\pi$, where $g$ is now finite of degree at least $2$.
It follows that $\pi(X(k))$ is a thin set. The result in this case thus
follows from Theorem \ref{thm:equi_thin}. 

We may therefore assume that the generic fibre of $\pi$ is geometrically integral.
Choose models $\mathcal{X}$ and $\mathcal{Y}$ for $X$ and $Y$ over $\fo_k$, together
with a map $\pi:\mathcal{Y} \to \mathcal{X}$ which restricts to the original map $\pi$ on $X$ and $Y$. Then  we clearly have
\begin{align*}
	N(U,H,\pi,B) &\leq \#\{ x \in U(k) : H(x) \leq B, \, x \in \pi(Y(k_\fp))\, \forall \, \fp\} \\
	&\leq \#\{ x \in U(k) : H(x) \leq B, \, x \bmod \fp^2 \in \pi(\mathcal{Y}(\fo_\fp)) \bmod \fp^2\, \forall \, \fp\}.
\end{align*}
Let $z > 0$. Imposing the above local conditions for all $\fp$ with $\Norm \fp \leq z$,
we may use equidistribution, Example \ref{ex:prod_measures}, and \eqref{eqn:Euler_product}, to obtain
$$
	\lim_{B \to \infty} \frac{N(U,H,\pi,B)}{N(U,H,B)} \ll_{X,H} \prod_{\Norm \fp \leq z} 
	\frac{\#(\pi(\mathcal{Y}(\oo_\fp)) \bmod \fp^2)}
			{\#\mathcal{X}(\oo_k/\fp^2)} \ll \frac{1}{(\log z)^{\Delta(\pi)}},
$$
where the implied constant is independent of $z$.
Our assumption that there is a non-pseudo-split fibre over some codimension $1$ point implies that $\Delta(\pi) > 0$.
Taking $z \to \infty$ completes the proof of Theorem \ref{thm:equi_pseudo_split}. \qed

\subsection{Proof of Theorem \ref{thm:equi_friable}}

We begin with the following result.

\begin{lemma} \label{lem:Z}
	Let $X$ be a finite type scheme over $\oo_k$ whose generic fibre $X_k$ is geometrically integral.
	Assume that $X(\FF_\fp) \neq \emptyset$ for all primes $\fp$ and let $Z \subset X$ a divisor which is flat over $\fo_k$.
	Then
	$$\prod_{\Norm \fp < z} \left( 1 - \frac{\#Z(\FF_\fp)}{\#X(\FF_\fp)} \right) \asymp (\log z)^{-r(Z)}, \quad z \to \infty,$$
	where $r(Z)$ denotes the number of irreducible components of $Z_k$.
\end{lemma}
\begin{proof}
Let $n=\dim X_k$. 
	By \cite[Cor.~7.13]{Ser12} we have
$$
		\sum_{\Norm \fp < z} \#Z(\FF_\fp) = r(Z) \frac{z^{n-1}}{\log( z^{n-1} )} + O\left( \frac{z^{n-1}}{(\log z )^2} \right).
$$
Note that  $\#X(\FF_\fp)=(\Norm \fp)^{n}+O((\Norm \fp)^{n-1/2})$ by Lang--Weil \cite{LW54}.
Hence, on 
	taking logarithms and 
	combining this with partial summation, we obtain
\begin{equation}\label{eq:flight}
	\begin{split}
		\log \prod_{\Norm \fp < z} \left( 1 - \frac{\#Z(\FF_\fp)}{\#X(\FF_\fp)} \right) 
		& = \sum_{\Norm \fp < z} \log  \left( 1 - \frac{\#Z(\FF_\fp)}{\#X(\FF_\fp)} \right)  \\
		& = -\sum_{\Norm \fp < z} \frac{\#Z(\FF_\fp)}{\#X(\FF_\fp)} + O(1)\\
&=
 - r(Z) \log \log z + O(1).
			\end{split}\end{equation}
	Exponentiating yields the result.
\end{proof}

Let now $X, \mathcal{X}, Z,\mathcal{Z}$ be as in Theorem \ref{thm:equi_friable} and let $z > y > 0$.  Example~\ref{ex:prod_measures} and
Lemma \ref{lem:Z} yield
\begin{align*}
	& \lim_{B \to \infty} \frac{\#\{x \in U(k): H(x) \leq B, ~x \text{ is $y$-friable with respect to } \mathcal{Z} \}}{N(U,H,B)} \\
	& \qquad\qquad  \ll \prod_{y< \Norm \fp < z} \left( 1 - \frac{\#\mathcal{Z}(\FF_\fp)}{\#\mathcal{X}(\FF_\fp)} \right)
	\ll \left(\frac{\log y}{\log z}\right)^{r(Z)}.
\end{align*}
Taking $z \to \infty$ completes the proof. \qed

\section{Zeros of quadratic forms in fixed residue classes}\label{s:modM}

Let $F\in \ZZ[x_1,\dots,x_n]$ be an isotropic quadratic form with non-zero discriminant $\Delta_F\in \ZZ$.
For any positive integer $M$ and each prime power factor $p^m\| M$ suppose that we are given  a non-empty subset 
\begin{equation}\label{eq:subset}
\Omega_{p^m}\subseteq \{\x\in (\ZZ/p^m\ZZ)^n: p\nmid \x, ~F(\x)\equiv 0\bmod{p^m}\}.
\end{equation}
Put $\Omega_M=\prod_{p^m\|M}\Omega_{p^m}$.
For $\x \in \ZZ^n$, we write $[\x]_M$ for its reduction modulo $M$.
In this section we shall
use the Hardy--Littlewood circle method to
produce an asymptotic formula for  the counting function
$$
\hat N(B,\Omega_M)
= \sum_{\substack{\x\in\ZZ^n,  F(\x)=0\\ [\x]_M\in \Omega_M}} w(\x/B),
$$
where $w:\RR^n\to \RR_{\geq  0}$ is an infinitely differentiable function with compact support. 

Associated to $F$ and $w$ is the weighted real density
 $\sigma_\infty(w)$, as defined in \cite[Thm.~3]{HB}. It satisfies $1\ll_{F,w} \sigma_\infty(w)\ll_{F,w} 1$.
Moreover, we have the 
associated  $p$-adic density 
\begin{equation}\label{eq:sigma-p}
\sigma_p=\lim_{k\to \infty} p^{-(n-1)k} \#\{\x \in (\ZZ/p^k\ZZ)^n: F(\x)\equiv 0\bmod{p^k} \},
\end{equation}
for each prime $p$. The goal of this section   is to prove the following result.

\begin{theorem}\label{t:circle}
Assume that $n\geq 5$ 
and that $\nabla F(\x)\gg 1$ for all $\x\in \supp(w)$.
Assume that 
 $M$ is coprime to $2\Delta_F$ 
and let $\Omega_M$ be as in \eqref{eq:subset}.
Then 
\begin{align*}
\hat N(B,\Omega_M)=~&
 \sigma_\infty(w)B^{n-2}
\prod_{p\nmid M}\sigma_p
\prod_{p^m\| M} \frac{\#\Omega_{p^m}}{p^{m(n-1)}}
+O_{\ve,F,w}\left(
B^{n/2+\ve}M^{n/2+\ve}\right), \, \forall \varepsilon > 0.
\end{align*}
\end{theorem}

In this result and henceforth in  this section, the implied constant is allowed to depend on the choice of $\ve$, the form $F$ and the weight function $w$, but not on the modulus $M$.
To ease notation we  shall suppress this dependence in what follows. 

Some comments are in order about the statement of this result. 
The condition that  $\nabla F(\x)\gg 1$ for any $\x$ in the  support of $w$
is required to simplify the analysis of certain oscillatory integrals  in the argument.  The  assumptions  $(M,2\Delta_F)=1$ and $(\x,M)=1$ for any $\x\in \Omega_M$ are made purely to simplify the expression for the leading constant in the asymptotic formula for 
$\hat N(B,\Omega_M)$.

It is possible to obtain a version of 
Theorem \ref{t:circle} by exploiting existing work in the literature, such as 
using work of Sardari \cite[Thm.~1.8]{sardari} to 
 handle the contribution from $\x\equiv \a \bmod{M}$,  for each $\a\in \Omega_M$.
However, this  leads to weaker results than our approach.
Nonetheless, several facets of Theorem~\ref{t:circle}  could still be improved. Firstly, one can do better in the $B$-aspect of the error term when $n$ is odd. Secondly,  it would not be hard to deal with the cases $n=3$ or $4$. Finally, 
when $M$ is square-free it is possible to improve the error term 
to $O(B^{n/2+\ve} \#\Omega_M^{1/2})$.
In order to simplify our exposition
we have not  pursued these improvements here. In 
 our application $\Omega_M$ will be comparable in size to the set
 of  $\x\in (\ZZ/M\ZZ)^n$ for which $F(\x)\equiv 0\bmod{M}$, 
leading us to relax the dependence on 
$\#\Omega_M$, often to the extent that we
employ the trivial inequality $\#\Omega_M\leq M^n$.

\subsection{First steps}
We begin the proof of Theorem	\ref{t:circle} by invoking the version of the circle method developed by Heath-Brown
 \cite[Thm.~1]{HB}. This implies that 
$$
\hat N(B,\Omega_M)=\frac{c_Q}{Q^2}
\sum_{q=1}^{\infty}\;
\sumstar_{a \bmod{q}}
\sum_{\substack{\x\in\ZZ^n\\
[\x]_M\in \Omega_M 
}} w(\x/B)
e_q(aF(\x)) h\left(\frac{q}{Q},\frac{F(\x)} {Q^2}\right),
$$
for any $Q>1$. Here $c_Q$ is a positive constant 
satisfying $c_Q=1+O_A(Q^{-A})$ for any $A>0$ and,  moreover,
$h(x,y)$ is a smooth
function  defined on the set $(0,\infty)\times\mathbb R$ such that 
 $h(x,y)\ll x^{-1}$ for all $y$, with
$h(x,y)$ non-zero only for $x\leq\max\{1,2|y|\}$. 
In particular, we are only interested in $q\ll Q$ in this sum.

We will henceforth take $Q=B$.  It is natural to break the sum into residue classes modulo the least common multiple $[q,M]$ and then apply Poisson summation,  as in the proof of \cite[Thm.~2]{HB}. This leads to the expression
$$
\hat N(B,\Omega_M)=\frac{c_B}{B^2}
\sum_{q\ll B}\; 
\sum_{\c\in \ZZ^n}
[q,M]^{-n}
S_{q,M}(\c)J_{q,M}(\c),
$$
where
\vspace{-10pt}
\begin{equation}\label{eq:SUM}
S_{q,M}(\c)=\sumstar_{a\bmod{q}}~
\sum_{\substack{\y\bmod{[q,M]}\\
[\y]_M\in \Omega_M }
} e_q\left(aF(\y)\right) e_{[q,M]}\left(\c.\y\right)
\end{equation}
and 
\vspace{-10pt}
\begin{align*}
J_{q,M}(\c)
&=\int_{\mathbb R^n}w(\x/B) 
h\left(\frac{q}{B},\frac{F(\x)}{B^2}\right)
e_{[q,M]}(-\c.\x)\d \x\\
&=B^n \int_{\mathbb R^n}w(\x) 
h\left(\frac{q}{B},F(\x)\right)
e_{[q,m]}(-B\c.\x)\d \x.
\end{align*}
For any $r>0$ and  $\v\in \RR^n$ it will be convenient to set 
\begin{equation}\label{eq:I*}
I_{r}^*(\v)
=\int_{\mathbb R^n}w(\x) 
h\left(r,F(\x)\right)
e_{r}(-\v.\x)\d \x.
\end{equation}
 In this notation, which coincides with that of
\cite[\S 7]{HB}, 
 we may clearly write
$
J_{q,M}(\c)= B^n I_{r}^*({M'}^{-1}\c),
$
where $r=q/B$ and 
$M'=[q,M]/q=M/(M,q)$.
Thus
\begin{equation}\label{eq:hat-N}
\hat N(B,\Omega_M)=c_B B^{n-2}
\sum_{q\ll B}\; 
\sum_{\c\in \ZZ^n}
[q,M]^{-n}
S_{q,M}(\c)I_{r}^*({M'}^{-1}\c).
\end{equation}

\subsection{The exponential sum}

In this section we analyse the sum 
$S_{q,M}(\c)$ in \eqref{eq:SUM} for $q,M\in \NN$ with $(M,2\Delta_F) =1$.
We begin by establishing the following.

\begin{lemma}\label{lem:fact}
Let $M=M_1M_2$.
Suppose that $(q_1M_1,q_2M_2)=1$ and choose integers $s,t$ such that 
$[q_1,M_1]s+[q_2,M_2]t=1$. Then 
$$
S_{q_1q_2,M}(\c)=S_{q_1,M_1}(t\c)S_{q_2,M_2}(s\c).
$$
 \end{lemma}

\begin{proof}
Note that $[q_1q_2,M]=[q_1,M_1][q_2,M_2]$. 
As $\y_1$ runs modulo $[q_1,M_1]$ and $\y_2$ runs modulo $[q_2,M_2]$, so $\y=\y_1 [q_2,M_2]t+\y_2 [q_1,M_1]s$ runs over  a full set of residue classes modulo $[q_1q_2,M]$.
Now let $\bar{q_1},\bar{q_2}\in \ZZ$ be such that 
$q_1\bar{q_1}+q_2\bar{q_2}=1$.  Then $a=a_1 q_2\bar{q_2}+a_2 q_1\bar{q_1}$ runs over 
$(\ZZ/q_1q_2\ZZ)^*$ as $a_1$ (resp.~$a_2$) runs over $(\ZZ/q_1\ZZ)^*$ (resp.~$(\ZZ/q_2\ZZ)^*$).
Under these transformations 
$[\y]_M\in \Omega_M $ $\Leftrightarrow$ 
$[\y_i]_{M_i}\in \Omega_{M_i} $ for $i=1,2$,
since $[q_1,M_1]s+[q_2,M_2]t=1$. Furthermore, 
\begin{align*}
e_{[q_1q_2,M]}\left(\c.\y\right)&=e_{[q_1,M_1]}\left(t\c.\y_1 \right)
e_{[q_2,M_2]}\left(s\c.\y_2 \right)
\end{align*}
and 
\begin{align*}
e_{q_1q_2}\left(aF(\y)\right)&=
e_{q_1}\left(a_1 \bar{q_2} F(\y)\right)e_{q_2}\left(a_2 \bar{q_1}F(\y)\right)\\
&=
e_{q_1}\left(a_1 \bar{q_2} ([q_2,M_2]t)^2 F(\y_1 )\right)e_{q_2}\left(a_2 \bar{q_1}([q_1,M_1]s)^2F(\y_2 )\right).
\end{align*}
Note that $(\bar{q_2} ([q_2,M_2]t)^2,q_1)=
(\bar{q_1} ([q_1,M_1]s)^2,q_2)=1$. 
A further change of variables in the $a_1$ and $a_2$ summations therefore proves the lemma.
\end{proof}

For any 
divisor $L\mid M$, we henceforth set 
$$
K_{L}(\c)=S_{1,L}(\c)=
\sum_{\y\in \Omega_{L}} e_L(\c.\y).
$$
While it is clear that $K_L(\0)=\#\Omega_L$,  we expect $K_L(\c)$ to be rather smaller than  $\#\Omega_L$
for typical values of $\c\in \ZZ^n$.  This will be established  in \S \ref{s:black}. 

Next, let 
$
S_q(\c)=S_{q,1}(\c). 
$
This is precisely the exponential sum appearing in \cite[Thm.~2]{HB}. 
Recall that the dual form $F^*\in \ZZ[\x]$ has underlying matrix $\Delta_F \mathbf{A}^{-1}$, where $\mathbf{A}$ is the symmetric matrix of determinant $\Delta_F$ that is associated to  $F$.
Our next result is  a variant of  \cite[Lem.~28]{HB}   and concerns the mean square.

\begin{lemma}\label{lem:average-HB}
Let $\ve>0$. Then 
$$
\sum_{\substack{q\leq R}} 
|S_{q}(\c)|^2\ll \begin{cases}
R^{n+3} & \text{ if $n$ is even and $F^*(\c)=0$,}\\
R^{n+5/2+\ve}(1+|\c|)^\ve & \text{ otherwise}.
\end{cases}
$$
\end{lemma}

\begin{proof}
The first bound follows directly from \cite[Lem.~25]{HB}.
As in the proof of 
\cite[Lem.~28]{HB}, we split $q$ into a square-free part $u$ and a square-full part $v$, finding  that 
$
|S_q(\c)|^2\ll u^{n+1+\ve}(u,F^*(\c))
v^{n+2},
$
where the factor $(u,F^*(\c))$ can be dropped if $n$ is odd.
Assuming that $n$ is odd or $F^*(\c)\neq 0$, it therefore follows that 
$$
\sum_{\substack{q\leq R}} 
|S_{q}(\c)|^2\ll  
\sum_{v\leq R}v^{n+2}  \left(\frac{R}{v}\right)^{n+2+\ve}(1+|\c|)^\ve
\ll R^{n+5/2+\ve}(1+|\c|)^\ve,
$$
since there are $O(R^{1/2})$ square-full values of $v\leq R$.
\end{proof}

Before returning to the exponential sum $S_{q,M}(\c)$ in 
\eqref{eq:SUM}, we first record the following estimate.

\begin{lemma}\label{lem:Weyl-T}
Let 
$a\in (\ZZ/q\ZZ)^*$, let $\c\in \ZZ^n$ and 
let $M\mid q$.
Then 
$$
\left|\sum_{\substack{\y\bmod{q}\\ [\y]_M\in \Omega_M }} e_q\left(aF(\y)+\c.\y\right)\right|
 \ll \frac{(qM)^{n/2}}{(q/M,M)^{n/2}}.
$$
\end{lemma}

\begin{proof}
Let $T_{q,M}(\c)$ denote the sum whose modulus is to be estimated.
Then 
$$
T_{q,M}(\c)=
\sum_{\u\in \Omega_M} 
\sum_{\substack{\y\bmod{q}\\ \y\equiv\u\bmod{M}}} e_q\left(aF(\y)+\c.\y\right).
$$
Let $q'=q/M$.
We make the change of variables  $\y=\u+M\z$ for $\z\bmod{q'}$, giving 
\begin{align*}
|T_{q,M}(\c)|
&\leq 
\sum_{\u\in \Omega_M} 
\left|
\sum_{\substack{\z\bmod{q'}}} e_{Mq'}\left(a\{M^2F(\z)+M\z.\nabla F(\u)\}+M\c.\z\right)\right|\\
&=
\sum_{\u\in \Omega_M} 
\left|
\sum_{\substack{\z\bmod{q'}}} e_{q'}\left(aMF(\z)+\z.\{\nabla F(\u)+\c\}\right)\right|.
\end{align*}
Let $q''=q'/h$, 
where $h=(q',M)$. We write $\mathbf{j}=\nabla F(\u)+\c$ for convenience. 
The next step is to make the change of variables $\z=\z_1+q''\z_2$ for $\z_1\bmod{q''}$ and $\z_2\bmod{h}$. 
Noting that  $q'\mid Mq''$, 
the inner sum is
\begin{align*}
\sum_{\substack{\z_1\bmod{q''}}}
\sum_{\substack{\z_2\bmod{h}}}
&e_{q'}\left(aMF(\z_1)+(\z_1+q''\z_2).\mathbf{j}	\right)\\
&
=
\begin{cases}
h^n
\sum_{\substack{\z_1\bmod{q''}}}
e_{q''}\left(ah^{-1}MF(\z_1)+h^{-1}\z_1.\mathbf{j}	\right) &\text{ if ${h\mid \mathbf{j}}$,}\\
0 & \text{ otherwise}.
\end{cases}
\end{align*}
When $h\mid \mathbf{j}$ the sum over $\z_1$ is $O({q''}^{n/2})$ by the proof of  \cite[Lem.~25]{HB}, since $(q'',h^{-1}M)=1$.
Thus
\begin{align*}
|T_{q,M}(\c)|
&\ll h^n {q''}^{n/2}
\#\left\{\u\in \Omega_M :2
\mathbf{A}\u\equiv -\c \bmod{h}\right\}\\
&\leq h^n {q''}^{n/2}
\#\left\{\u\in (\ZZ/M\ZZ)^n :2
\mathbf{A}\u\equiv -\c \bmod{h}\right\},
\end{align*}
where $\mathbf{A}$ is the matrix associated to $F$. 
As
 $\mathbf{A}$ is non-singular, the inner cardinality is 
$O((M/h)^n)$. We conclude the proof on recalling that 
$q''=q/(hM)$.
\end{proof}

We now return to the exponential sum $S_{q,M}(\c)$ in 
\eqref{eq:SUM}.
There is a unique factorisation into pairwise coprime positive integers $u,v_1,v_2$, with
$v_1$ square-free  and $v_2$ square-full, 
 such that 
\begin{equation}\label{eq:q-factor}
q=uv_1v_2, \quad \text{ with $(u,M)=1$ and $v_1v_2\mid M^{\infty}$.}
\end{equation}
Likewise there is a unique factorisation $M=M_{11}M_{12}M_2$, where
\begin{equation}\label{eq:M-factor}
M_{1i}=(M,v_i) \quad \text{ and } \quad 
M_2=\frac{M}{M_{11}M_{12}}.
\end{equation} 
It follows that $M_{11}=v_1$, since $v_1$ is square-free
and $v_1\mid M^\infty$.
Moreover, we have  $(M_2,uv_1v_2)=1$ and 
$$
[q,M]=\frac{uv_1v_2M_{11}M_{12}M_2}{(v_1v_2,M_{11}M_{12})}=uv_1v_2M_2.
$$
We may now establish the following factorisation of  
$S_{q,M}(\c)$.

\begin{lemma}\label{lem:3way-split}
We have 
$$
S_{q,M}(\c)=
\phi(v_1)S_{u}(\c)
S_{v_2,M_{12}}(\overline{uv_1M_2}\c) 
K_{v_1}(\overline{uv_2M_2}\c)
K_{M_2}(\overline{uv_1v_2}\c), 
$$
where $\overline{uv_1M_2}\in \ZZ$ (respectively, 
$\overline{uv_2M_2}\in \ZZ$, $\overline{uv_1v_2}\in \ZZ$)  is a multiplicative inverse of 
 $uv_1M_2$ (respectively, 
$uv_2M_2$, $uv_1v_2$) modulo 
 $v_2$ (respectively, 
$v_1$, $M_2$).
\end{lemma}

\begin{proof}
We write $v=v_1v_2$ and $M_1=M_{11}M_{12}$ for convenience. 
Let  $\overline{uM_2}\in \ZZ$ (respectively,  $\overline{uv}\in \ZZ$) be such that 
$uM_2\overline{uM_2}\equiv 1\bmod{v}$ 
(respectively,  $uv\overline{uv}\equiv 1\bmod{M_2}$).
Then the factorisation
$$
S_{q,M}(\c)=S_{u}(\c)S_{v,M_1}(\overline{uM_2}\c)
K_{M_2}(\overline{uv}\c)
$$
is a direct consequence of Lemma \ref{lem:fact} and 
 the obvious fact that 
$S_q(t\c)=S_q(\c)$, for any  $(t,q)=1$.
Note that since $v_i\mid M_{1i}$, we have 
$[v_i,M_{1i}]=v_i$, for $i=1,2$.
A further application of Lemma \ref{lem:fact} now yields
$$
S_{v_1v_2,M_{11}M_{12}}(\mathbf{d})=
S_{v_1,v_1}(t\mathbf{d})S_{v_2,M_{12}}(s\mathbf{d}),
$$
where  $s,t\in \ZZ$ are such that 
$v_1s+v_2t=1$. Finally, we note that  
$$
S_{v_1,v_1}(t\mathbf{d})
=\sumstar_{a\bmod{v_1}}
\sum_{\substack{\y\in \Omega_{v_1}}
} e_{v_1}\left(aF(\y)+t\mathbf{d}. \y\right)=
\phi(v_1)K_{v_1}(t\mathbf{d}),
$$
since $v_1\mid F(\y)$ for any $\y\in \Omega_{v_1}$.
This completes the proof of the lemma.
\end{proof}

The following simple upper bound will suffice to handle the third factor in the  factorisation of
Lemma \ref{lem:3way-split}.

\begin{lemma}\label{lem:Weyl}
For any $\c\in \ZZ^n$, we have 
$
S_{v_2,M_{12}}(\c)\ll 
v_2^{n/2+1}M_{12}^{n/2}.
$
\end{lemma}

\begin{proof}
Recall the definition \eqref{eq:SUM} of 
$S_{v_2,M_{12}}(\c)$.  
We shall sum trivially over $a$.  Noting that $M_{12}\mid v_2$, the statement follows on 
applying Lemma \ref{lem:Weyl-T} to estimate the inner sum over $\y$ and noting that $(v_2,M_{12}^2)\geq M_{12}$.
\end{proof}

\subsection{Contribution from the trivial character}

Returning to \eqref{eq:hat-N}, the contribution $T(B)$, say, from the term $\c=\0$ is found to satisfy
\begin{equation}\label{eq:TB}
T(B)= \left(1+O_A(B^{-A})\right)B^{n-2}
\sum_{q\ll B}
[q,M]^{-n}S_{q,M}(\0)
I_{q/B}^*(\0),
\end{equation}
for any $A>0$, 
in the notation of \eqref{eq:SUM} and \eqref{eq:I*}.

Recall that $n\geq 5$.
We shall start by analysing an unweighted version of the sum over $q$ in \eqref{eq:TB}.
For any prime $p$, recall the definition \eqref{eq:sigma-p} of 
the   $p$-adic density $\sigma_p$.
In particular the product  $\prod_{p}\sigma_p$ is absolutely convergent for $n\geq 5$.

\begin{lemma}\label{lem:sing-series}
Let $R\leq B$.
Then 
$$
\sum_{q\leq R}
[q,M]^{-n}S_{q,M}(\0)= 
\prod_{p\nmid M}\sigma_p
\prod_{p^m\| M} \frac{\#\Omega_{p^m}}{p^{m(n-1)}}
+
O\left(R^{(3+\kappa-n)/2}B^{\ve} M^{(n-1-\kappa)/2+\ve} \right),
$$
for any $\ve>0$, 
where
\begin{equation}\label{eq:kappa}
\kappa=\begin{cases}
1 &\text{ if $2\mid n$,}\\
0 &\text{ if $2\nmid n$.}
\end{cases}
\end{equation}
\end{lemma}

\begin{proof}
We make 
 the change of variables $q=uv_1v_2$ 
and $M=M_{11}M_{12}M_2$, in the notation of \eqref{eq:q-factor} and 
\eqref{eq:M-factor}. Recalling that $[q,M]=uv_1v_2M_2$, 
an application of 
Lemma \ref{lem:3way-split} implies that 
\begin{align*}
\sum_{q\leq R}
[q,M]^{-n}
S_{q,M}(\0)
&=
\sum_{\substack{v_1v_2\leq R\\ v_1v_2\mid M^{\infty} }}\sum_{\substack{u\leq R/(v_1v_2)\\ 
(u,M)=1}} \frac{\phi(v_1)S_{u}(\0)S_{v_2,M_{12}}(\0)\#\Omega_{v_1}\#\Omega_{M_2} }{(uv_1v_2M_2)^{n}}.
\end{align*}
Since $R\leq B$, an inspection of the proof of  \cite[Lemmas 28 and 31]{HB} reveals that 
$$
\sum_{\substack{u\leq R/(v_1v_2)\\ 
(u,M)=1}} 
u^{-n}
S_{u}(\0) =\prod_{p\nmid M} \sigma_p +O\left((R/v_1v_2)^{(3+\kappa-n)/2}B^\ve M^{\ve/2}\right).
$$
Lemma \ref{lem:Weyl} yields
$S_{v_2,M_{12}}(\0)\ll v_2^{n/2+1}M_{12}^{n/2}$. Thus the  
overall contribution from the error term is 
\begin{align*}
&\ll R^{(3+\kappa-n)/2}B^{\ve}M^{\ve/2}
\sum_{\substack{v_1v_2\leq R\\ v_1v_2\mid M^{\infty} }}
\frac{(v_1v_2)^{(n-3-\kappa)/2}
v_1v_2^{n/2+1} M_{12}^{n/2}\#\Omega_{v_1} \#\Omega_{M_2} }{(v_1v_2M_2)^{n}} \\
&\ll R^{(3+\kappa-n)/2}B^{\ve}M^{\ve/2}
\sum_{\substack{v_1v_2\leq R\\ v_1v_2\mid M^{\infty} }}
\frac{ M_{12}^{n/2}v_1^{(n-1-\kappa)/2}  }{v_2^{(1+\kappa)/2}},
\end{align*}
since 
$\#\Omega_{v_1}\leq v_1^n$ and 
$\#\Omega_{M_2}\leq M_2^n$.
Recalling that $M_{12}\mid v_2$, we next observe that
$
M_{12}^{n/2}/v_2^{(1+\kappa)/2}\leq 
M_{12}^{(n-1-\kappa)/2}.
$
Since
$$
\#\{v\leq R: v\mid M^\infty\} \leq \sum_{v\mid M^\infty} \left(\frac{R}{v}\right)^\ve =
R^\ve \prod_{p\mid M}\left(1-p^{-\ve}\right)^{-1}\ll R^{\ve}M^{\ve/4},
$$
the error term gives the overall contribution 
$
O(  R^{(3+\kappa-n)/2}B^{3\ve}M^{(n-1-\kappa)/2+\ve}).
$

Redefining the choice of $\ve>0$, we have therefore proved that 
\begin{align*}
\sum_{q\leq R}
[q,M]^{-n}
S_{q,M}(\0)
=~&
\prod_{p\nmid M}\sigma_p
\sum_{\substack{v_1v_2\leq R\\ v_1v_2\mid M^{\infty} }}
\frac{\phi(v_1)S_{v_2,M_{12}}(\0)\#\Omega_{v_1}\#\Omega_{M_2} }{
(v_1v_2M_2)^{n}}\\ \quad &+
O\left(R^{(3+\kappa-n)/2}B^\ve M^{(n-1-\kappa)/2+\ve}\right).
\end{align*}
Employing  Lemma \ref{lem:Weyl} once more, we obtain
\begin{align*}
\sum_{\substack{v_1v_2>R\\ v_1v_2\mid M^{\infty} }}
\frac{\phi(v_1)\#\Omega_{v_1}\#\Omega_{M_2} S_{v_2,M_{12}}(\0)}{
(v_1v_2M_2)^{n}}
&\ll 
\sum_{\substack{v_1v_2> R\\ v_1v_2\mid M^\infty}} \frac{M_{12}^{n/2}}{v_2^{n/2-1}}
\left(\frac{v_1v_2}{R}\right)^{(n-3-\kappa)/2-\ve}
\\
&\ll R^{(3+\kappa-n)/2}B^{\ve} M^{(n-1-\kappa)/2+\ve},
\end{align*}
on recalling that $v_1=M_{11}$ and arguing as before.
In view of the fact  that  $\prod_{p\nmid M}\sigma_p\ll 1$ for $n\geq 5$, it  follows that 
\begin{align*}
\sum_{q\leq R}
[q,M]^{-n}
S_{q,M}(\0)
=~&
C_M
\prod_{p\nmid M}\sigma_p
+
O\left(R^{(3+\kappa-n)/2}B^\ve M^{(n-1-\kappa)/2+\ve}\right).
\end{align*} 
Here
$$
C_M=\sum_{\substack{v_1v_2\mid M^{\infty} }}
\frac{\phi(v_1)\#\Omega_{v_1}\#\Omega_{M_2} S_{v_2,M_{12}}(\0)}{
(v_1v_2M_2)^{n}},
$$
where the sum is constrained to satisfy  $(v_1,v_2)=\mu^2(v_1)=1$, with  $v_2$ is square-full.
It remains to calculate this quantity.
We have
$$
\sum_{\substack{v_1\mid M^\infty\\ (v_1,v_2)=1}} \phi(v_1)\mu^2(v_1) = \frac{M^\flat}{(M,v_2)^\flat},
$$
where $k^\flat=\prod_{p\mid k}p$ is the square-free kernel.
Since $\#\Omega_{v_1}\#\Omega_{M_2}=\#\Omega_{M}/\#\Omega_{M_{12}}$, we see that
\begin{align*}
C_M&=\frac{{\#\Omega_M} M^\flat}{M^n}
\sum_{\substack{v_2\mid M^{\infty} \\ \text{$v_2$ square-full}}}
\frac{S_{v_2,M_{12}}(\0)M_{12}^n}{
(M,v_2)^\flat v_2^{n} \#\Omega_{M_{12}}}\\
&=
\frac{{\#\Omega_M} M^\flat}{M^n}
\prod_{p^m\|M} \left(1+\frac{1}{p}\sum_{\ell\geq 2} 
\frac{S_{p^\ell,p^{\min\{\ell,m\}}}(\0) p^{\min\{\ell,m\}n}}{
p^{\ell n} \#\Omega_{p^{\min\{\ell,m\}}}}\right).
\end{align*}
Let $\Sigma_p$ denote the sum over $\ell$. 
The contribution to $\Sigma_p$ from $\ell\in [2,m]$ is 
$
\sum_{2\leq \ell\leq m}  \phi(p^\ell)
 =p^m-p.
$
On the other hand, on evaluating the Ramanujan sum,  the contribution to $\Sigma_p$ from $\ell>m$ is 
\begin{align*}
\frac{p^{mn}}{\#\Omega_{p^m}}
\sum_{\ell>m} 
\frac{S_{p^\ell,p^{m}}(\0) }{p^{\ell n}}
&=
\frac{p^{mn}}{\#\Omega_{p^m}}
\sum_{\ell>m} 
\frac{p^\ell N(\ell)-p^{n+\ell-1}N(\ell-1)}{p^{\ell n}}\\
&=
\frac{p^{mn}}{\#\Omega_{p^m}}
\left(\lim_{\ell\to \infty} 
p^{-\ell(n-1)} N(\ell) -p^{-m(n-1)}\#\Omega_{p^m}\right),
\end{align*}
where 
$N(k)=\#\{\y\bmod{p^k}: p^k\mid F(\y), ~[\y]_{p^m} \in \Omega_{p^m}\}$,
for any $k\geq m$.
Note that $p\nmid \y$ in $N(k)$ by our assumption on 
$\Omega_M$ in \eqref{eq:subset}.
Putting everything together, we deduce that 
$$
C_M=
\frac{{\#\Omega_M} M^\flat}{M^n}
\prod_{p^m\|M} 
\left(\frac{p^{mn-1}}{\#\Omega_{p^m}}
\lim_{\ell\to \infty} 
p^{-\ell(n-1)} N(\ell)\right)=
\prod_{p^m\|M} 
\lim_{\ell\to \infty} 
p^{-\ell(n-1)} N(\ell).
$$
Finally, since $(M,2\Delta_F)=1$, a straightforward application of
Lemma \ref{lem:Hensel}
shows that 
$N(k+1)=p^{n-1}N(k)$ for all $k\geq m$. Thus  we can replace the limit by $p^{-m(n-1)}N(m)=
p^{-m(n-1)}\#\Omega_{p^m}$.
This  completes the proof of the lemma. 
\end{proof}

For convenience we put 
$$
\mathfrak{S}(M)=
\prod_{p\nmid M}\sigma_p
\prod_{p^m\| M} \frac{\#\Omega_{p^m}}{p^{m(n-1)}}.
$$

Next, 
let $\sigma_\infty(w)$ be the weighted real density associated to $F$, as defined in \cite[Thm.~3]{HB}. 
Since $\nabla F(\x)\gg 1$ throughout the support of $w$, it follows from \cite[Lem.~13]{HB} that
$
I_{q/B}^*(\0)
=\sigma_\infty(w)+O_A((q/B)^A),
$
for any $A>0$. 
Hence
$$
\sum_{q\leq B^{1-\ve}}
\frac{S_{q,M}(\0)
I_{q/B}^*(\0)}{[q,M]^{n}} = 
\sigma_\infty(w)
\sum_{q\leq B^{1-\ve}}
[q,M]^{-n}S_{q,M}(\0)
 +O(B^{-n}),
$$
on taking $A$ sufficiently large. 
We apply Lemma \ref{lem:sing-series} with $R=B^{1-\ve}$ to estimate the inner sum, finding that 
\begin{align*}
\sum_{q\leq B^{1-\ve}}
\frac{S_{q,M}(\0)
I_{q/B}^*(\0)}{[q,M]^{n}} =~& \sigma_\infty(w)\mathfrak{S}(M)+O\left(
B^{(3+\kappa-n)/2+(n+1)\ve}M^{(n-1-\kappa)/2+\ve}
\right).
\end{align*}

We have  the bounds 
$
\frac{\partial^i }{\partial q^i} I_{q/B}^*(\0)\ll q^{-i}$, for $i\in \{0,1\}$,
which are a direct consequence of \cite[Lemmas 14 and 15]{HB}.
Hence 
we may combine Lemma \ref{lem:sing-series} with partial summation to conclude that 
\begin{align*}
\sum_{B^{1-\ve}<q\ll B}
\frac{S_{q,M}(\0)
I_{q/B}^*(\0) }{[q,M]^{n}}
&\ll 
B^{(3+\kappa-n)/2+(n+1)\ve}M^{(n-1-\kappa)/2+\ve}.
\end{align*}
Bringing everything together in \eqref{eq:TB}, 
and redefining the choice of $\ve$,  
we finally arrive at the  estimate
$$
T(B)=
\sigma_\infty(w)\mathfrak{S}(M)B^{n-2}
+O\left(
B^{(n-1+\kappa)/2+\ve}M^{(n-1-\kappa)/2+\ve}\right).
$$
This shows that the contribution from the trivial character is satisfactory for Theorem \ref{t:circle}.

\subsection{Contribution from the non-trivial characters}\label{s:black}

It remains to consider the contribution $E(B)$, say,  from  
$\c\neq \0$ in \eqref{eq:hat-N}.
Thus
$$
E(B)\ll 
B^{n-2}
\sum_{q\ll B}\; 
\sum_{\0\neq \c\in \ZZ^n}
[q,M]^{-n}
|S_{q,M}(\c)| |I_{r}^*({M'}^{-1}\c)|,
$$
where $r=q/B$,  
$M'=[q,M]/q=M/(q,M)$ and 
$I_{r}^*(\v)$ is defined in \eqref{eq:I*}.
It follows from \cite[Lemmas 14 and 18]{HB}
that 
$
I_r^*(\v)\ll_A r^{-1}|\v|^{-A}
$
for any $A>0$. Hence there is a negligible contribution to $E(B)$ from  vectors $\c\in \ZZ^n$ for which $|\c|\geq M' B^\ve$ for any fixed value of $\ve>0$. 
We now apply \cite[Lemmas 14 and 22]{HB} to deduce that
$$
I_r^*(\v)\ll  (r^{-2}|\v|)^{\ve/10} (r^{-1}|\v|)^{1-n/2}.
$$
Hence 
\begin{align*}\label{eq:EB}
E(B)
&\ll 
B^{n/2-1+\ve}
\sum_{\substack{q\ll B}} 
\sum_{\substack{\c\in \ZZ^n\\ 
0<|\c|\leq M'B^{\ve}
}}|\c|^{1-n/2}
\frac{|S_{q,M}(\c)|}{[q,M]^{n/2+1}},
\end{align*}
since $qM'=[q,M]$ and $r^{-1}|\v|=B|\c|/[q,M]$.
We carry out the change of variables 
recorded in \eqref{eq:q-factor} and \eqref{eq:M-factor} and recall that 
$\#\Omega_{v_1}\leq v_1^n$.
In this notation, it follows from
Lemmas~\ref{lem:3way-split}   and \ref{lem:Weyl} that 
\begin{align*}
\frac{|S_{q,M}(\c)|}{[q,M]^{n/2+1}}
&=
\frac{\left| \phi(v_1)S_{u}(\c)
S_{v_2,M_{12}}(\overline{uv_1M_2}\c) 
K_{v_1}(\overline{uv_2M_2}\c)
K_{M_2}(\overline{uv_1v_2}\c)\right|}
{(uv_1v_2M_2)^{n/2+1}}\\
&\ll 
\frac{\left|S_{u}(\c)
v_1^{n/2}
K_{M_2}(\overline{uv_1v_2}\c) M_{12}^{n/2}\right| }{(uM_2)^{n/2+1}},	
\end{align*}
where  $M_2=M/(M,v_1v_2)$.

Let $\mathcal{V}$ denote the set of  vectors $(v_1,v_2)\in \NN^2$ such that $v_1v_2\ll B$ and $v_1v_2\mid M^\infty$, with 
$(v_1,v_2)=\mu^2(v_1)=1$ and $v_2$ square-full. 
Noting that $M'=M/(q,M)=M_2$, 
 we deduce that
\begin{equation}\label{eq:EB}
E(B)
\ll
B^{n/2-1+\ve}
\sum_{\substack{(v_1,v_2)\in \mathcal{V} }}
\frac{v_1^{n/2}M_{12}^{n/2}
}{M_2^{n/2+1}}
E_{v_1,v_2}(B),
\end{equation}
where
$$
E_{v_1,v_2}(B)=
\sum_{\substack{u\ll B/(v_1v_2)\\ 
(u,M)=1}}
\sum_{\substack{\c\in \ZZ^n\\ 
0<|\c|\leq M_2B^{\ve}
}}|\c|^{1-n/2}
\frac{\left|S_{u}(\c)
K_{M_2}(\overline{uv_1v_2}\c)\right|}{u^{n/2+1}}.
$$
The presence of $\bar{u}$ in 
$K_{M_2}(\overline{uv_1v_2}\c)$ prevents us from 
executing the sum over $u$ directly. 
To separate 
$S_{u}(\c)$ and $K_{M_2}(\overline{uv_1v_2}\c)$, we shall apply  Cauchy's inequality. This gives
$E_{v_1,v_2}(B)^2\leq  \Sigma_1\Sigma_2$,
where
\begin{align*}
\Sigma_1&=
\sum_{\substack{u\ll B/(v_1v_2)\\ 
(u,M)=1}}
\sum_{\substack{\c\in \ZZ^n\\ 
0<|\c|\leq M_2B^{\ve}
}}
|\c|^{2-n}
\frac{\left|S_{u}(\c)\right|^2}{u^{n+2}},\\
\Sigma_2&=
\sum_{\substack{u\ll B/(v_1v_2)\\ 
(u,M)=1}}
\sum_{\substack{\c\in \ZZ^n\\ 
0<|\c|\leq M_2B^{\ve}
}}
\left|K_{M_2}(\overline{uv_1v_2}\c)\right|^2.
\end{align*}
The following results are concerned with estimating these quantities. 

\begin{lemma}\label{lem:sigma1}
We have 
$$
\Sigma_1\ll  (M_2B)^\ve \left(
M_2^2 \left(B/(v_1v_2)\right)^{1/2}+
\left(B/(v_1v_2)\right)^{(1+\kappa)/2}\right),
$$
where $\kappa$ is given by \eqref{eq:kappa}.
\end{lemma}

\begin{proof}
Let $v=v_1v_2$.
To begin with, 
it follows from Lemma \ref{lem:average-HB} and partial summation that 
\begin{align*}
\sum_{\substack{u\ll B/v}}
\frac{\left|S_{u}(\c)\right|^2}{u^{n+2}}\ll 
\begin{cases}
\left(B/v\right)^{1/2+\ve}|\c|^\ve, &\text{ if $F^*(\c)\neq 0$,}\\
\left(B/v\right)^{(1+\kappa)/2+\ve}(1+|\c|)^\ve, 
&\text{ if $F^*(\c)= 0$}.
\end{cases}
\end{align*}
A standard estimate shows that 
there are $O(C^{n-2})$ vectors $\c\in \ZZ^n$, such that  $|\c|\leq C$ and $F^*(\c)=0$. 
Hence 
\begin{align*}
\Sigma_1\ll  (M_2B^\ve)^{2+\ve}
\left(B/v\right)^{1/2+\ve}+
(M_2B^\ve)^{2\ve}
\left(B/v\right)^{(1+\kappa)/2+\ve},
\end{align*}
on breaking the $\c$-sum into dyadic intervals. This therefore completes the proof of the lemma, on redefining the choice of $\ve>0$.
\end{proof}

\begin{lemma}
\label{lem:sigma2}
We have 
$
\Sigma_2\ll  B^{1+\ve}M_2^{2n}/(v_1v_2).
$
\end{lemma}

\begin{proof}
Since $K_{M_2}(\overline{uv_1v_2}\c)$ only depends on the value of $\c$ modulo $M_2$, we may break into residue classes modulo $M_2$, concluding that
\begin{align*}
\Sigma_2
&\leq 
\sum_{\substack{u\ll B/(v_1v_2)}}
\sum_{\a\bmod{M_2}}
\left|K_{M_2}(\a)\right|^2
\#\left\{
\c\in \ZZ^n: 
|\c|\leq M_2B^{\ve}, ~\c\equiv \a\bmod{M_2}\right\}\\
&\ll \frac{B^{1+n\ve}}{v_1v_2}
\sum_{\a\bmod{M_2}}
\left|K_{M_2}(\a)\right|^2.
\end{align*}
But the inner sum over $\a$ is
$M_2^n \#\Omega_{M_2}\leq M_2^{2n}$, 
by   orthogonality of characters.  The lemma follows on redefining $\ve$.
\end{proof}

Combining Lemma \ref{lem:sigma1} and \ref{lem:sigma2} in \eqref{eq:EB}, we deduce that 
\begin{align*}
E(B)
&\ll
B^{(n-1)/2+2\ve} 
\sum_{\substack{(v_1,v_2)\in \mathcal{V} }}
\frac{v_1^{(n-1)/2} M_{12}^{n/2}M_2^{n/2+\ve}}{v_2^{1/2}}
 \left(
\frac{B}{v_1v_2}\right)^{(1+\kappa)/4}\\
&\ll
B^{n/2+(\kappa-1)/4+2\ve}  M^{n/2+\ve}
\sum_{\substack{(v_1,v_2)\in \mathcal{V} }}
\frac{1}{(v_1v_2)^{(3+\kappa)/4}}\\
&\ll 
B^{n/2+(\kappa-1)/4+3\ve} M^{n/2+2\ve},
\end{align*}
since $v_1M_{12}M_2=M$.
This completes the proof of Theorem \ref{t:circle} on redefining the choice of $\ve>0$. \qed

\section{The  Selberg sieve on quadrics}
\label{s:weights}

In this section we prove Theorem \ref{thm:Selberg} and its applications. 

\subsection{Proof of Theorem \ref{thm:Selberg}}

Points of $X(\QQ)$ are represented by vectors $\x=(x_0,\dots,x_n)\in \ZZ_{\mathrm{prim}}^{n+1}$ such that 
$F(\x)=0$, where $F \in \ZZ[x_0,\ldots,x_n]$ is the quadratic form defining $X$.  As in the previous section,  
for any $d\in \NN$ we write $[\x]_d$ for the reduction of  $\x$ modulo $d$. Passing to the affine cone, we have   
$$
 N(X,H,\Omega, B)
\leq \#\left\{\x\in \ZZ_{\text{prim}}^{n+1}: |\x|\leq B,~ F(\x)=0, ~
[\x]_{p^m} \in \hat\Omega_{p^m} \text{ for all $p$}\right\}.
$$
where  $\hat\Omega_{p^m}=\{\x\in (\ZZ/p^m\ZZ)^{n+1}: p\nmid \x, ~
(x_0:\dots:x_n)\in \Omega_{p^m}
\}$
and $|\cdot|$ is the supremum norm on $\RR^{n+1}$.
We apply the Selberg sieve to  estimate this.

Consider the function $\omega_0: \RR \rightarrow \RR_{\geq 0}$, given by
$$
\omega_0(x) = 
\begin{cases}
e^{-(1-x^2)^{-1}} & \text{ if $|x| < 1$,}\\  
0 & \text{ if $|x|\geq 1$}.
\end{cases} 
$$
Then $\omega_0$ is infinitely differentiable and compactly supported on
$[-1,1]$.   We work with the weight function
$w:\RR^{n+1}\to \RR_{\geq 0}$, given by 
$$
w(\x)=\omega_0\left(\tfrac{5}{2}|\A\x|-2\right),
$$
where $\A$ is the non-singular matrix defining $F$, with determinant $\Delta_F$.  
It is clear that $w(\x)=0$ unless $\tfrac{2}{5}\leq |\A\x| \leq \frac{6}{5}$. In particular $w$ is supported on a region $\x\ll 1$, where we adhere to the convention that the implied constant in any estimate is allowed to depend on $F$.
Moreover, $\nabla F(\x)\geq \tfrac{1}{5}$ 
throughout the support of $w$.  It therefore follows that $w$ belongs to the class of weight functions $\mathcal{C}_0^+(S)$ introduced in 
\cite[\S 2 and \S 6]{HB}, for an appropriate set of parameters $S$ including $n$ and  the coefficients of the  quadratic form $F$.

We have 
$|\A\x|\leq (n+1)\|\A\|B$ for any $\x\in \ZZ^{n+1}$ such that $|\x|\leq B$, where $\|\A\|$ is the maximum modulus of 
 the coefficients of $\A$. Let $c=(n+1)\|\A\|$. We break the sum into dyadic intervals for $|\A\x|$, finding that
\begin{align*}
 N(X,H,\Omega, B)
&\leq 
\sum_{j=0}^\infty 
\#\left\{\x\in \ZZ_{\text{prim}}^{n+1}:  
\begin{array}{l}
2^{-j-1}cB<|\A\x|\leq 2^{-j}cB\\ 
F(\x)=0,
~[\x]_{p^m}\in \hat\Omega_{p^m} \text{ for all $p$}
\end{array}
\right\}\\
&\ll 
\sum_{j=0}^\infty 
\sum_{\substack{\x\in\ZZ_{\text{prim}}^{n+1}, ~F(\x)=0\\ [\x]_{p^m}\in \hat \Omega_{p^m} 
\text{ for all $p$}}}
w(2^j\x/(cB)).
\end{align*}
It will clearly suffice to show that 
\begin{equation}\label{eq:guides}
\sum_{\substack{\x\in\ZZ_{\text{prim}}^{n+1}, ~F(\x)=0\\ [\x]_{p^m}\in \hat \Omega_{p^m} 
\text{ for all $p$}}}
w(\x/B)
\ll_{\ve,X}
\frac{B^{n-1}}{G(\xi)}
+\xi^{m(n+1)+2+\ve} B^{(n+1)/2+\ve},
\end{equation}
for any $B,\xi\geq 1$, 
with 
$G(\xi)$ as in the statement of 
Theorem \ref{thm:Selberg}.

Let $P$ denote the produce over distinct primes $p<\xi$ for which $\omega_p>0$.
For $n\in \NN$ we define the finite sequence of non-negative numbers
$$
a_n=\sum_{\substack{\x\in\ZZ_{\text{prim}}^{n+1} \\ F(\x)=0\\ n(\x)=n}}
w(\x/B), \quad \text{ where } \quad
n(\x)=\prod_{\substack{p\mid P\\ 
[\x]_{p^m}\in \hat \Omega_{p^m}^c 
}}p.
$$
The left hand  side of \eqref{eq:guides} can be written
$\sum_{(n,P)=1}a_n$. We seek to apply the Selberg sieve, in the form \cite[Thm.~7.1]{FI}, to estimate this quantity.

Let $d\mid P$. 
First  note that $d\mid n(\x)$ if and only if $[\x]_{p^m}\in \hat \Omega_{p^m}^c $ for all $p\mid d$, for any $\x$ appearing in the definition of $a_n$.
Hence    Theorem \ref{t:circle} implies that 
$$
\sum_{d\mid n}a_n= 
g(d)B^{n-1} \sigma_\infty(w)
\prod_{p}\sigma_p 
+O_{\ve,X}(d^{m(n+1)/2+\ve/4}B^{(n+1)/2+\ve}),
$$
where 
$$
g(d)=\prod_{p\mid d} \left(1-\frac{\#\hat \Omega_{p^m}}{\#\hat X(\ZZ/p^m)}\right)=\prod_{p\mid d} \omega_p,
$$
in the notation of \eqref{eq:omega}.
In deriving this expression for $g(d)$,  we have used Lemma~\ref{lem:Hensel} for 
$p\nmid 2\Delta_F$ to usher in the appearance of  $
\#\hat X(\ZZ/p^m\ZZ)$ in the denominator.
Clearly $g(p)=\omega_p$ satisfies $0<g(p)<1$ for every $p\mid P$.
It now follows from \cite[Thm.~7.1 and Eq.~(7.32)]{FI} that 
$$
\sum_{(n,P)=1}a_n\ll_{\ve,X} \frac{B^{n-1}}{G(\xi)} + \sum_{d\leq \xi^2} \tau_3(d)
d^{m(n+1)/2+\ve/4}B^{(n+1)/2+\ve},
$$
Taking the trivial bound $\tau_3(d)\ll d^{\ve/4}$ and summing over $d\leq \xi^2$, this therefore concludes the proof of \eqref{eq:guides} and so the proof of 
Theorem \ref{thm:Selberg}.  \qed

\subsection{Proof of Theorem \ref{thm:quadrics_thin}}\label{s:thin}

Let $X\subset \PP^n$ be a non-singular quadric hypersurface defined over $\QQ$, with dimension  $n-1\geq 3$. Let $\Upsilon\subset X(\QQ)$ be a thin subset.  To prove Theorem \ref{thm:quadrics_thin}, it suffices to consider thin sets of type $I$ and $II$ (see \S \ref{sec:thin}).

We begin with the more difficult case of type $II$.
By  Lemma \ref{prop:thin} there is a
set of primes $\cP$ of positive natural density $\delta$
and a constant $c\in (0,1)$, such
that for each $p \in \cP$ we have 
$$\#(\Upsilon \bmod p)\leq cp^{n-1} + O_{\Upsilon}(p^{n-3/2}).$$
Taking $m=1$ in 
\eqref{eq:omega}, 
for such $p$ we therefore have
$\omega_p \geq 1-c+O_{\Upsilon}(p^{-1/2})$. 
It follows that there exists 
$\eta<(1-c)/c$ such that 
$$\frac{\omega_p}{1 -\omega_p} \geq 
\eta$$
for large enough $p\in \cP$. Let $\cP^\circ$ denote the set of such $p \in \cP$.
An application of Lemma \ref{lem:Wirsing} 
now yields
$$
G(\xi) \geq  \sum_{\substack{a \leq \xi \\ p\mid a \Rightarrow p \in \mathcal{P}^\circ}} \mu^2(a)  \eta^{\omega(a)}  \gg_{\Upsilon,X} \xi(\log \xi)^{\eta\delta -1}  \gg_{\varepsilon,\Upsilon,X} \xi^{1-\varepsilon},
$$
for any $\varepsilon > 0$.
It therefore follows from Theorem~\ref{thm:Selberg} that 
$$
\#\left\{x\in \Upsilon(\QQ): H(x)\leq B\right\}
\ll_{\ve,\Upsilon,X}
\xi^{-1+\ve}B^{n-1} + \xi^{n+3+\ve}B^{(n+1)/2+\ve},
$$
Balancing the terms by choosing  $\xi=B^{\theta_n}$, with 
$\theta_n=\frac{n-3}{2(n+4)}=\frac{1}{2}-\frac{7}{2(n+4)}$,  
this is plainly satisfactory for Theorem 
\ref{thm:quadrics_thin}.

Turning to thin sets of type $I$, we 
let $Z \subset X$ be a Zariski closed subset with $Z \neq X$.
For any prime $p$,  Lemma \ref{prop:thin} implies that 
$\#Z(\FF_p) \leq cp^{n-2},$ 
for some $c=c(Z)>0$. Then $\omega_p \geq 1-cp^{-1}$ and it follows that 
$$\frac{\omega_p}{1 -\omega_p} \geq \frac{1 - cp^{-1}}{cp^{-1}} = \frac{p}{c} - 1.$$
A further application of  Lemma \ref{lem:Wirsing} now implies that 
$
G(\xi) \gg_{\varepsilon,\Upsilon} \xi^{2 - \varepsilon}
$
for all $\ve > 0$. We complete the proof of Theorem 
\ref{thm:quadrics_thin} by arguing as above. \qed

\subsection{Proof of Theorem \ref{thm:quadrics_Delta}} 
\label{s:arne}

Let $\pi:Y\to X$ be a dominant map with $X\subset \PP^n$ a smooth quadric hypersurface of dimension at least $3$, as in the statement of  
Theorem \ref{thm:quadrics_Delta}. Then 
\begin{align*}
	N(X,H,\pi,B) 
	& \leq \# \{ x \in X(\QQ): H(x) \leq B, ~x \in \pi(Y(\QQ_p)) ~\forall p\} \\
	& \leq \# \{ x \in X(\QQ): H(x) \leq B, ~x \bmod p^2 \in \pi(Y(\ZZ_p)) \bmod p^2 ~\forall p\}.
\end{align*}
We now apply Theorem \ref{thm:Selberg} with $\Omega_{p^2} = (\pi(Y(\ZZ_p)) \bmod p^2)$ to find that 
$$
N(X,H,\pi,B)\ll_{\ve,X}
 \frac{B^{n-1}}{G(\xi)} + \xi^{2n +4 + \varepsilon}B^{(n+1)/2 + \varepsilon},
 $$
where 
$$
G(\xi)=\sum_{a \leq \xi} \mu^2(a )\prod_{p\mid a} 
\left(\frac{\omega_p}{1-\omega_p}\right), \quad  \omega_p = 1 - \frac{\#\pi(Y(\ZZ_p) \bmod p^2)}{\#X(\ZZ/p^2\ZZ)}.$$
Taking $\xi$ to be a small power of $B$, the result follows from Corollary \ref{cor:Wirsing}. \qed

\subsection{Proof of Theorem \ref{thm:quadrics_friable}}\label{s:fry}
Let $Z\subset X$ be a divisor and let $r(Z)$ be the number of irreducible components of $Z$.
We are led to apply Theorem \ref{thm:Selberg} with $m=1$ and 
$\Omega_p=X(\FF_p)\setminus Z(\FF_p)$ for $p>y$. 
Thus 
\begin{equation}\label{eq:ORD}
\omega_p=  
1 - \frac{\#X(\FF_p) - \#Z(\FF_p)}{\#X(\FF_p)}=
\frac{ \#Z(\FF_p)}{\#X(\FF_p)}.
\end{equation}
As $1-\omega_p\leq 1$,  we obtain 
\begin{align*}
G(\xi)
& \geq \sum_{\substack{k<\xi\\ p\mid k\Rightarrow p> y}}
\mu^2(k)\prod_{p\mid k} \omega_p.
\end{align*}
Applying \eqref{eq:flight} and  Lemma \ref{lem:Wirsing}, we see that 
\begin{align*}
  \sum_{\substack{k<\xi\\ p\mid k\Rightarrow p> y}}
k \mu^2(k) \prod_{p\mid k} \omega_p
& \asymp_y\frac{\xi}{\log \xi} \prod_{p < \xi}\left(1 + \omega_p \right) 
=\frac{\xi}{\log \xi} \prod_{p < \xi}\frac{1 + \omega_p^2 }{
1-\omega_p}.
\end{align*}
As $\omega_p$ is given by \eqref{eq:ORD},  
an application of Lemma \ref{lem:Z} shows that the product over primes is 
$\asymp (\log \xi)^{r(Z)}.$ Using partial summation to remove  $k$, we have
$G(\xi)\gg (\log \xi )^{r(Z)}.$
The result follows on
taking $\xi$ to be a  small power of $B$.
\qed

\end{document}